\documentclass[12pt]{article}
\usepackage{a4wide,latexsym,amsfonts,amssymb,exscale,enumerate,amscd}
\usepackage{amsthm}
\usepackage{
  amsmath,
  graphicx, 
  times,
  color,
  hyphenat,
  mathptmx,
  pinlabel
}

\newcommand{\figins}[3] 
{\raisebox{#1pt}{\includegraphics[height=#2 in]{#3}}}

\usepackage{pstricks}
\usepackage[tiling]{pst-fill}

\psset{linewidth=0.3pt,dimen=middle} \psset{xunit=.70cm,yunit=0.70cm}
\psset{arrowsize=1pt 5,arrowlength=0.6,arrowinset=0.7}

\usepackage[all]{xy}
\SelectTips{cm}{}

\usepackage{graphicx}

\newcommand{\sseq}{{\rm SSeq}}

\newcommand{\ii}{ \textbf{\textit{i}}}
\newcommand{\jj}{ \textbf{\textit{j}}}

\newcommand{\onel}{{\mathbf 1}_{\lambda}}
\newcommand{\onelp}{{\mathbf 1}_{\lambda'}}

\newcommand{\refequal}[1]{\xy {\ar@{=}^{#1}
(-1,0)*{};(1,0)*{}};
\endxy}

\newcommand{\FF}{\mathbf{2Foam}_{/\ell}}

\newcommand{\Ucat}{\cal{U}_k}

\newcommand{\Uup}{\xy {\ar (0,-3)*{};(0,3)*{} };(0,0)*{\bullet};(2,0)*{};(-2,0)*{};\endxy}
\newcommand{\Udown}{\xy {\ar (0,3)*{};(0,-3)*{} };(0,0)*{\bullet};(2,0)*{};(-2,0)*{};\endxy}
\newcommand{\Ucupr}{\xy (-2,1)*{}; (2,1)*{} **\crv{(-2,-3) & (2,-3)} ?(1)*\dir{>}; \endxy}
\newcommand{\Ucupl}{\xy (2,1)*{}; (-2,1)*{} **\crv{(2,-3) & (-2,-3)}?(1)*\dir{>};\endxy}
\newcommand{\Ucapr}{\xy (-2,-1)*{}; (2,-1)*{} **\crv{(-2,3) & (2,3)}?(1)*\dir{>};\endxy\;\;}
\newcommand{\Ucapl}{\xy (2,-1)*{}; (-2,-1)*{} **\crv{(2,3) &(-2,3) }?(1)*\dir{>};\endxy\;}
\newcommand{\Ucross}{\xy {\ar (2.5,-2.5)*{};(-2.5,2.5)*{}}; {\ar (-2.5,-2.5)*{};(2.5,2.5)*{} };
(4,0)*{};(-4,0)*{};\endxy}
\newcommand{\Ucrossd}{\xy {\ar (2.5,2.5)*{};(-2.5,-2.5)*{}}; {\ar (-2.5,2.5)*{};(2.5,-2.5)*{} };
(4,0)*{};(-4,0)*{};\endxy}

\newcommand{\BOX}{\hbox {$\sqcap$ \kern -1em $\sqcup$}}

\newcommand{\To}{\Rightarrow}

\newcommand{\Hom}{{\rm Hom}}

\newcommand{\HOMU}{{\rm HOM_{\Ucat}}}
\renewcommand{\to}{\rightarrow}
\newcommand{\maps}{\colon}

\newcommand{\scs}{\scriptstyle}

\theoremstyle{definition}
\newtheorem{thm}{Theorem}[section]

\newtheorem{lem}[thm]{Lemma}

\newtheorem{defn}[thm]{Definition}

        \newcommand{\be}{\begin{equation}}
        \newcommand{\ee}{\end{equation}}
        \newcommand{\ba}{\begin{eqnarray}}
        \newcommand{\ea}{\end{eqnarray}}
        \newcommand{\ban}{\begin{eqnarray*}}
        \newcommand{\ean}{\end{eqnarray*}}
        \newcommand{\barr}{\begin{array}}
        \newcommand{\earr}{\end{array}}

\numberwithin{equation}{section}

\def\emph#1{{\sl #1\/}}

\let\phi=\varphi
\let\epsilon=\varepsilon

\usepackage{bbm}

\def\N{{\mathbbm N}}
\def\R{{\mathbbm R}}
\def\Z{{\mathbbm Z}}

\def\F{{\mathbbm F}}

\def\cal#1{\mathcal{#1}}%
\def\1{\mathbbm{1}}%
\def\nn{\notag}

\def\mf{\mathfrak}
\def\shuffle{\,\raise 1pt\hbox{$\scriptscriptstyle\cup{\mskip
               -4mu}\cup$}\,}

\newcommand{\HH}{\mathcal{H}}
\newcommand{\lowrru}[1]{\xybox{%
  (-8,0)*{};
  (8,0)*{};
  (-6,-18)*{};(6,-9)*{} **\crv{(-6,-13) & (6,-15)} ?(1)*\dir{>};
  (6,-9)*{};(6,0)*{}  **\dir{-} ?(.3)*\dir{ }+(2,0)*{\scs {\bf j}};
}}

\newcommand{\lowllu}[1]{\xybox{%
  (-8,0)*{};
  (8,0)*{};
  (6,-18)*{};(-6,-9)*{} **\crv{(6,-13) & (-6,-15)} ?(1)*\dir{>};
  (-6,-9)*{};(-6,0)*{}  **\dir{-} ?(.3)*\dir{ }+(-2,0)*{\scs {\bf j}};
}}

\newcommand{\bbe}[1]{\xybox{%
  (-2,0)*{};
  (2,0)*{};
  (0,0);(0,-18) **\dir{-}; ?(.5)*\dir{<}+(2.3,0)*{\scriptstyle{#1}};
}}

\newcommand{\bbsid}{\xybox{%
  (-2,0)*{};
  (2,0)*{};
  (0,10);(0,4) **\dir{-};
}}
\newcommand{\bbpef}[1]{\xybox{%
  (-6,0)*{};
  (6,0)*{};
  (-4,0)*{}="t1";
  (4,0)*{}="t2";
  "t1";"t2" **\crv{(-4,-6) & (4,-6)}; ?(.15)*\dir{>} ?(.9)*\dir{>}
   ?(.5)*\dir{}+(0,-2)*{\scriptstyle{#1}};
}}
\newcommand{\bbpfe}[1]{\xybox{%
  (-6,0)*{};
  (6,0)*{};
  (-4,0)*{}="t1";
  (4,0)*{}="t2";
  "t2";"t1" **\crv{(4,-6) & (-4,-6)}; ?(.15)*\dir{>} ?(.9)*\dir{>}
  ?(.5)*\dir{}+(0,-2)*{\scriptstyle{#1}};
}}

\newcommand{\bbcfe}[1]{\xybox{%
  (-6,0)*{};
  (6,0)*{};
  (-4,0)*{}="t1";
  (4,0)*{}="t2";
  "t1";"t2" **\crv{(-4,6) & (4,6)}; ?(.15)*\dir{>} ?(.9)*\dir{>}
  ?(.5)*\dir{}+(0,2)*{\scriptstyle{#1}};
}}
\newcommand{\bbcef}[1]{\xybox{%
  (-6,0)*{};
  (6,0)*{};
  (-4,0)*{}="t1";
  (4,0)*{}="t2";
  "t2";"t1" **\crv{(4,6) & (-4,6)}; ?(.15)*\dir{>}
  ?(.9)*\dir{>} ?(.5)*\dir{}+(0,2)*{\scriptstyle{#1}};
}}

\newcommand{\ccbub}[2]{
\xybox{%
 (-6,0)*{};
  (6,0)*{};
  (-4,0)*{}="t1";
  (4,0)*{}="t2";
  "t2";"t1" **\crv{(4,6) & (-4,6)}; ?(.7)*\dir{}+(-2,0)*{\scs #2}
  ?(.05)*\dir{>} ?(1)*\dir{>};
  "t2";"t1" **\crv{(4,-6) & (-4,-6)};
   ?(.3)*\dir{}+(0,0)*{\bullet}+(0,-3)*{\scs {#1}};
}}
\newcommand{\cbub}[2]{
\xybox{%
 (-6,0)*{};
  (6,0)*{};
  (-4,0)*{}="t1";
  (4,0)*{}="t2";
  "t2";"t1" **\crv{(4,6) & (-4,6)};?(.7)*\dir{}+(-2,0)*{\scs #2};
   ?(0)*\dir{<} ?(.95)*\dir{<};
  "t2";"t1" **\crv{(4,-6) & (-4,-6)};
   ?(.3)*\dir{}+(0,0)*{\bullet}+(0,-3)*{\scs {#1}};
}}

\newcommand{\bbdl}[1]{\xybox{%
  (2,0);(0,-8) **\crv{(2,-2)&(0,-6)}; ?(.5)*\dir{>}
}}
\newcommand{\bbdlu}[1]{\xybox{%
  (2,0);(0,-8) **\crv{(2,-2)&(0,-6)}; ?(.5)*\dir{<}
}}
\newcommand{\bbdr}[1]{\xybox{%
  (-2,0);(0,-8) **\crv{(-2,-2)&(0,-6)}; ?(.5)*\dir{>}
}}
\newcommand{\bbdru}[1]{\xybox{%
  (-2,0);(0,-8) **\crv{(-2,-2)&(0,-6)}; ?(.5)*\dir{<}
}}

\title{$\mathfrak{sl}_3$-Foams and the Khovanov-Lauda categorification of 
quantum $\mathfrak{sl}_k$.}
\author{Marco Mackaay}

\begin{document}
\maketitle

%
\begin{abstract} 
In this paper I define certain interesting 2-functors from the 
Khovanov-Lauda 2-category which 
categorifies quantum $\mathfrak{sl}_k$, for any $k>1$, to a 2-category 
of universal $\mathfrak{sl}_3$ foams with corners. For want of a better name 
I use the term {\em foamation} to indicate those 2-functors. I conjecture 
the existence of similar 2-functors to the 2-category of 
$\mathfrak{sl}_n$ foams with corners, for any $n>1$. 
\end{abstract}
%
%
%
%
\section{Introduction}\label{sec:intro}

In this paper I relate the Khovanov-Lauda (KL) categorification of quantum 
$\mathfrak{sl}_k$, for $k>1$, defined in \cite{KL}, 
to the universal $\mathfrak{sl}_3$-foams, defined in \cite{MV}. Since both 
can be defined in terms of 2-categories, the most natural formulation of the 
relation is as a 2-functor, which I call {\em foamation}, because it produces 
a foam for each KL-diagram. As a matter of fact there are several such 
2-functors, all very similar, depending on a finite parameter. 

Unfortunately there are some signs which do not work out well under 
foamation, so I have defined everything over $\F_2$. Hopefully this 
problem can be solved in the future. If not, one wonders if 
foamation leads to a slightly different categorification of quantum 
$\mathfrak{sl}_k$. 

With $\mathfrak{sl}_2$-foams, defined by Bar-Natan~\cite{B-N}, foamation 
also works, but there is not much to check, since the image of the KL-diagrams 
is mostly zero. I conjecture that foamation also works with 
$\mathfrak{sl}_n$-foams, for $n\geq 4$, 
using the Kapustin-Li formula as in~\cite{MSV}.  
I have checked some relations for $n=4$ and $n=5$ and they are 
preserved by foamation. To 
generalize this to higher $n$ one would have to compute the Kapustin-Li 
formula for foams with facets with arbitrary thickness, which is a 
computational challenge in its own right. 

The foamation 2-functors should be related to the representation 
2-functors $\Gamma_N$ in~\cite{KL}. If foamation exists for all $n\geq 2$, 
then it should be a faithful family of 2-functors, meaning that for $n$ 
big enough foamation is faithful on all homogeneous elements of a fixed 
degree. The latter conjecture also shows an interesting feature of foamation: 
if true, it would allow one, in principal, to find all the relations in the 
categorification 
of quantum $\mathfrak{sl}_k$ just from the Kapustin-Li formula. Since it is 
also the Kapustin-Li formula which, in a certain sense, gives rise to the 
$\mathfrak{sl}_n$ link homologies~\cite{MSV}, a nice unifying picture seems 
to arrise.        

\section{The 2-category $\Ucat$}
\label{sec:catquantumslk}

In this section we recall the definition of Khovanov and Lauda's 
categorification $\Ucat$ of the quantum groups $U_q(\mathfrak{sl}_k)$, for any $k>1$. 
For motivation and more details we refer to their paper~\cite{KL}. 
The 2-category $\Ucat$ has the structure of an additive $\F_2$-linear
2-category. Thus between
any two objects there is an $\F_2$-linear Hom category, and 
composition and identities are given by additive $\F_2$-linear functors. 
%
%

From now on let $k\in\N_{>1}$ be arbitrary but fixed. 
In the sequel we use {\em signed sequences} $\ii=(\pm i_1,\ldots,\pm i_m)$, 
for any $m\in\N$ and any $i_j\in\{1,\ldots,k-1\}$. The set of signed sequences 
we denote $\sseq$. For $\ii=(\epsilon_1i_1,\ldots,\epsilon_mi_m)\in\sseq$ we 
define $\ii_X:=\epsilon_1 (i_1)_X+\cdots+\epsilon_m (i_m)_X$, where 
$$(i_j)_X=(0,0,\ldots,-1,2,-1,0\ldots,0),$$
such that the vector starts with $i_j-1$ and ends with $k-2-i_j$ zeros. We also 
define the symmetric $\Z$-valued bilinear form on $\{1,\ldots, k-1\}$ 
by $i\cdot i=2$, $i\cdot (i+1)=-1$ and $i\cdot j=0$, for $\vert i-j\vert>1$.

\begin{defn} \label{def_Ucat} $\Ucat$ is an
additive $\F_2$-linear 2-category. The 2-category $\Ucat$ consists of
\begin{itemize}
  \item objects: $\lambda=(\lambda_1,\ldots,\lambda_{k-1})$, with 
$\lambda_i\in\Z$.
\end{itemize}
The homs $\Ucat(\lambda,\lambda')$ between two objects $\lambda$, $\lambda'$ are
additive $\F_2$-linear categories consisting of:
\begin{itemize}
  \item objects\footnote{We refer to objects of the category
$\Ucat(\lambda,\lambda')$ as 1-morphisms of $\Ucat$.  Likewise, the morphisms of
$\Ucat(\lambda,\lambda')$ are called 2-morphisms in $\Ucat$. } of
$\Ucat(\lambda,\lambda')$: a 1-morphism in $\Ucat$ from $\lambda$ to $\lambda'$
is a formal finite direct sum of 1-morphisms
  \[
 \cal{E}_{\ii} \onel\{t\} =\onelp \cal{E}_{\ii} \onel\{t\}
  \]
for any $t\in \Z$ and signed sequence $\ii\in\sseq$ such that 
$\lambda'=\lambda+\ii_X$.
  \item morphisms of $\Ucat(\lambda,\lambda')$: for 1-morphisms $\cal{E}_{\ii} \onel\{t\}
  ,\cal{E}_{\jj} \onel\{t'\} \in \Ucat$, hom
sets $\Ucat(\cal{E}_{\ii} \onel\{t\},\cal{E}_{\jj} \onel\{t'\})$ of
$\Ucat(\lambda,\lambda')$ are graded $\F_2$-vector spaces given by linear
combinations of degree $t-t'$ diagrams, modulo certain relations, built from
composites of:
\begin{enumerate}[i)]
  \item  Degree zero identity 2-morphisms $1_x$ for each 1-morphism $x$ in
$\Ucat$; the identity 2-morphisms $1_{\cal{E}_{+i} \onel}\{t\}$ and
$1_{\cal{E}_{-i} \onel}\{t\}$, for $i \in I$, are represented graphically by
\[
\begin{array}{ccc}
  1_{\cal{E}_{+i} \onel\{t\}} &\quad  & 1_{\cal{E}_{-i} \onel\{t\}} \\ \\
    \xy
 (0,8);(0,-8); **\dir{-} ?(.5)*\dir{>}+(2.3,0)*{\scriptstyle{}};
 (0,-11)*{ i};(0,11)*{ i};
 (6,2)*{ \lambda};
 (-8,2)*{ \lambda +i_X};
 (-10,0)*{};(10,0)*{};
 \endxy
 & &
 \;\;   \xy
 (0,8);(0,-8); **\dir{-} ?(.5)*\dir{<}+(2.3,0)*{\scriptstyle{}};
 (0,-11)*{ i};(0,11)*{i};
 (6,2)*{ \lambda};
 (-8,2)*{ \lambda -i_X};
 (-12,0)*{};(12,0)*{};
 \endxy
\\ \\
   \;\;\text{ {\rm deg} 0}\;\;
 & &\;\;\text{ {\rm deg} 0}\;\;
\end{array}
\]
and more generally, for a signed sequence $\ii=\epsilon_1i_1 \epsilon_2i_2 \dots
\epsilon_mi_m$, the identity $1_{\cal{E}_{\ii} \onel\{t\}}$ 2-morphism is
represented as
\begin{equation*}
\begin{array}{ccc}
  \xy
 (-12,8);(-12,-8); **\dir{-};
 (-4,8);(-4,-8); **\dir{-};
 (4,0)*{\cdots};
 (12,8);(12,-8); **\dir{-};
 (-12,11)*{i_1}; (-4,11)*{ i_2};(12,11)*{ i_m };
  (-12,-11)*{ i_1}; (-4,-11)*{ i_2};(12,-11)*{ i_m};
 (18,2)*{ \lambda}; (-20,2)*{ \lambda+\ii_X};
 \endxy
\end{array}
\end{equation*}
where the strand labelled $i_{\alpha}$ is oriented up if $\epsilon_{\alpha}=+$
and oriented down if $\epsilon_{\alpha}=-$. We will often place labels with no
sign on the side of a strand and omit the labels at the top and bottom.  The
signs can be recovered from the orientations on the strands.

\item For each $\lambda \in X$ the 2-morphisms
\[
\begin{tabular}{|l|c|c|c|c|}
\hline
 {\bf Notation:} \xy (0,-5)*{};(0,7)*{}; \endxy&
 $\Uup_{i,\lambda}$  &  $\Udown_{i,\lambda}$  &$\Ucross_{i,j,\lambda}$
 &$\Ucrossd_{i,j,\lambda}$  \\
 \hline
 {\bf 2-morphism:} &   \xy
 (0,7);(0,-7); **\dir{-} ?(.75)*\dir{>}+(2.3,0)*{\scriptstyle{}}
 ?(.1)*\dir{ }+(2,0)*{\scs i};
 (0,-2)*{\txt\large{$\bullet$}};
 (6,4)*{ \lambda};
 (-8,4)*{ \lambda +i_X};
 (-10,0)*{};(10,0)*{};
 \endxy
 &
     \xy
 (0,7);(0,-7); **\dir{-} ?(.75)*\dir{<}+(2.3,0)*{\scriptstyle{}}
 ?(.1)*\dir{ }+(2,0)*{\scs i};
 (0,-2)*{\txt\large{$\bullet$}};
 (-6,4)*{ \lambda};
 (8,4)*{ \lambda +i_X};
 (-10,0)*{};(10,9)*{};
 \endxy
 &
   \xy
  (0,0)*{\xybox{
    (-4,-4)*{};(4,4)*{} **\crv{(-4,-1) & (4,1)}?(1)*\dir{>} ;
    (4,-4)*{};(-4,4)*{} **\crv{(4,-1) & (-4,1)}?(1)*\dir{>};
    (-5,-3)*{\scs i};
     (5.1,-3)*{\scs j};
     (8,1)*{ \lambda};
     (-12,0)*{};(12,0)*{};
     }};
  \endxy
 &
   \xy
  (0,0)*{\xybox{
    (-4,4)*{};(4,-4)*{} **\crv{(-4,1) & (4,-1)}?(1)*\dir{>} ;
    (4,4)*{};(-4,-4)*{} **\crv{(4,1) & (-4,-1)}?(1)*\dir{>};
    (-6,-3)*{\scs i};
     (6,-3)*{\scs j};
     (8,1)*{ \lambda};
     (-12,0)*{};(12,0)*{};
     }};
  \endxy
\\ & & & &\\
\hline
 {\bf Degree:} & \;\;\text{  $i\cdot i$ }\;\;
 &\;\;\text{  $i\cdot i$}\;\;& \;\;\text{  $-i \cdot j$}\;\;
 & \;\;\text{  $-i \cdot j$}\;\; \\
 \hline
\end{tabular}
\]

\[
\begin{tabular}{|l|c|c|c|c|}
\hline
  {\bf Notation:} \xy (0,-5)*{};(0,7)*{}; \endxy&\text{$\Ucupr_{i,\lambda}$} & \text{$\Ucupl_{i,\lambda}$} & \text{$\Ucapl_{i,\lambda}$} &
 \text{$\Ucapr_{i,\lambda}$} \\
 \hline
  {\bf 2-morphism:} &  \xy
    (0,-3)*{\bbpef{i}};
    (8,-5)*{ \lambda};
    (-12,0)*{};(12,0)*{};
    \endxy
  & \xy
    (0,-3)*{\bbpfe{i}};
    (8,-5)*{ \lambda};
    (-12,0)*{};(12,0)*{};
    \endxy
  & \xy
    (0,0)*{\bbcef{i}};
    (8,5)*{ \lambda};
    (-12,0)*{};(12,0)*{};
    \endxy
  & \xy
    (0,0)*{\bbcfe{i}};
    (8,5)*{ \lambda};
    (-12,0)*{};(12,0)*{};
    \endxy\\& & &  &\\ \hline
 {\bf Degree:} & \;\;\text{  $1+\lambda_i$}\;\;
 & \;\;\text{ $1-\lambda_i$}\;\;
 & \;\;\text{ $1+\lambda_i$}\;\;
 & \;\;\text{  $1-\lambda_i$}\;\;
 \\
 \hline
\end{tabular}
\]
\end{enumerate}

\item The $\mathfrak{sl}_2$ relations:
\begin{enumerate}[i)]
\item  $\mathbf{1}_{\lambda+i_X}\cal{E}_{+i}\onel$ and
$\onel\cal{E}_{-i}\mathbf{1}_{\lambda+i_X}$ are biadjoint, up to grading shifts:
\begin{equation} \label{eq_biadjoint1}
  \xy   0;/r.18pc/:
    (-8,0)*{}="1";
    (0,0)*{}="2";
    (8,0)*{}="3";
    (-8,-10);"1" **\dir{-};
    "1";"2" **\crv{(-8,8) & (0,8)} ?(0)*\dir{>} ?(1)*\dir{>};
    "2";"3" **\crv{(0,-8) & (8,-8)}?(1)*\dir{>};
    "3"; (8,10) **\dir{-};
    (12,-9)*{\lambda};
    (-6,9)*{\lambda+i_X};
    \endxy
    \; =
    \;
\xy   0;/r.18pc/:
    (-8,0)*{}="1";
    (0,0)*{}="2";
    (8,0)*{}="3";
    (0,-10);(0,10)**\dir{-} ?(.5)*\dir{>};
    (5,8)*{\lambda};
    (-9,8)*{\lambda+i_X};
    \endxy
\qquad \quad  \xy   0;/r.18pc/:
    (-8,0)*{}="1";
    (0,0)*{}="2";
    (8,0)*{}="3";
    (-8,-10);"1" **\dir{-};
    "1";"2" **\crv{(-8,8) & (0,8)} ?(0)*\dir{<} ?(1)*\dir{<};
    "2";"3" **\crv{(0,-8) & (8,-8)}?(1)*\dir{<};
    "3"; (8,10) **\dir{-};
    (12,-9)*{\lambda+i_X};
    (-6,9)*{ \lambda};
    \endxy
    \; =
    \;
\xy   0;/r.18pc/:
    (-8,0)*{}="1";
    (0,0)*{}="2";
    (8,0)*{}="3";
    (0,-10);(0,10)**\dir{-} ?(.5)*\dir{<};
   (9,8)*{\lambda+i_X};
    (-6,8)*{ \lambda};
    \endxy
\end{equation}

\begin{equation}
\label{eq_biadjoint2}
 \xy   0;/r.18pc/:
    (8,0)*{}="1";
    (0,0)*{}="2";
    (-8,0)*{}="3";
    (8,-10);"1" **\dir{-};
    "1";"2" **\crv{(8,8) & (0,8)} ?(0)*\dir{>} ?(1)*\dir{>};
    "2";"3" **\crv{(0,-8) & (-8,-8)}?(1)*\dir{>};
    "3"; (-8,10) **\dir{-};
    (12,9)*{\lambda};
    (-5,-9)*{\lambda+i_X};
    \endxy
    \; =
    \;
      \xy 0;/r.18pc/:
    (8,0)*{}="1";
    (0,0)*{}="2";
    (-8,0)*{}="3";
    (0,-10);(0,10)**\dir{-} ?(.5)*\dir{>};
    (5,-8)*{\lambda};
    (-9,-8)*{\lambda+i_X};
    \endxy
\qquad \quad \xy  0;/r.18pc/:
    (8,0)*{}="1";
    (0,0)*{}="2";
    (-8,0)*{}="3";
    (8,-10);"1" **\dir{-};
    "1";"2" **\crv{(8,8) & (0,8)} ?(0)*\dir{<} ?(1)*\dir{<};
    "2";"3" **\crv{(0,-8) & (-8,-8)}?(1)*\dir{<};
    "3"; (-8,10) **\dir{-};
    (12,9)*{\lambda+i_X};
    (-6,-9)*{ \lambda};
    \endxy
    \; =
    \;
\xy  0;/r.18pc/:
    (8,0)*{}="1";
    (0,0)*{}="2";
    (-8,0)*{}="3";
    (0,-10);(0,10)**\dir{-} ?(.5)*\dir{<};
    (9,-8)*{\lambda+i_X};
    (-6,-8)*{ \lambda};
    \endxy
\end{equation}
\item
\begin{equation} \label{eq_cyclic_dot}
    \xy
    (-8,5)*{}="1";
    (0,5)*{}="2";
    (0,-5)*{}="2'";
    (8,-5)*{}="3";
    (-8,-10);"1" **\dir{-};
    "2";"2'" **\dir{-} ?(.5)*\dir{<};
    "1";"2" **\crv{(-8,12) & (0,12)} ?(0)*\dir{<};
    "2'";"3" **\crv{(0,-12) & (8,-12)}?(1)*\dir{<};
    "3"; (8,10) **\dir{-};
    (15,-9)*{ \lambda+i_X};
    (-12,9)*{\lambda};
    (0,4)*{\txt\large{$\bullet$}};
    (10,8)*{\scs };
    (-10,-8)*{\scs i};
    \endxy
    \quad = \quad
      \xy
 (0,10);(0,-10); **\dir{-} ?(.75)*\dir{<}+(2.3,0)*{\scriptstyle{}}
 ?(.1)*\dir{ }+(2,0)*{\scs };
 (0,0)*{\txt\large{$\bullet$}};
 (-6,5)*{ \lambda};
 (8,5)*{ \lambda +i_X};
 (-10,0)*{};(10,0)*{};(-2,-8)*{\scs i};
 \endxy
    \quad = \quad
    \xy
    (8,5)*{}="1";
    (0,5)*{}="2";
    (0,-5)*{}="2'";
    (-8,-5)*{}="3";
    (8,-10);"1" **\dir{-};
    "2";"2'" **\dir{-} ?(.5)*\dir{<};
    "1";"2" **\crv{(8,12) & (0,12)} ?(0)*\dir{<};
    "2'";"3" **\crv{(0,-12) & (-8,-12)}?(1)*\dir{<};
    "3"; (-8,10) **\dir{-};
    (15,9)*{\lambda+i_X};
    (-12,-9)*{\lambda};
    (0,4)*{\txt\large{$\bullet$}};
    (-10,8)*{\scs };
    (10,-8)*{\scs i};
    \endxy
\end{equation}
\item  All dotted bubbles of negative degree are zero. That is,
\begin{equation} \label{eq_positivity_bubbles}
 \xy
 (-12,0)*{\cbub{\alpha}{i}};
 (-8,8)*{\lambda};
 \endxy
  = 0
 \qquad
  \text{if $\alpha<\lambda_i-1$} \qquad
 \xy
 (-12,0)*{\ccbub{\alpha}{i}};
 (-8,8)*{\lambda};
 \endxy = 0\quad
  \text{if $\alpha< -\lambda_i-1$}
\end{equation}
for all $\alpha \in \Z_+$, where a dot carrying a label $\alpha$ denotes the
$\alpha$-fold iterated vertical composite of $\Uup_{i,\lambda}$ or
$\Udown_{i,\lambda}$ depending on the orientation.  A dotted bubble of degree
zero equals 1:
\[
\xy 0;/r.18pc/:
 (0,0)*{\cbub{\lambda_i-1}{i}};
  (4,8)*{\lambda};
 \endxy
  = 1 \quad \text{for $\lambda_i \geq 1$,}
  \qquad \quad
  \xy 0;/r.18pc/:
 (0,0)*{\ccbub{-\lambda_i-1}{i}};
  (4,8)*{\lambda};
 \endxy
  = 1 \quad \text{for $\lambda_i \leq -1$.}
\]
\item For the following relations we employ the convention that all summations
are increasing, so that $\sum_{f=0}^{\alpha}$ is zero if $\alpha < 0$.
\begin{eqnarray}
\label{eq:redtobubbles}
  \text{$\xy 0;/r.18pc/:
  (14,8)*{\lambda};
  (-3,-8)*{};(3,8)*{} **\crv{(-3,-1) & (3,1)}?(1)*\dir{>};?(0)*\dir{>};
    (3,-8)*{};(-3,8)*{} **\crv{(3,-1) & (-3,1)}?(1)*\dir{>};
  (-3,-12)*{\bbsid};
  (-3,8)*{\bbsid};
  (3,8)*{}="t1";
  (9,8)*{}="t2";
  (3,-8)*{}="t1'";
  (9,-8)*{}="t2'";
   "t1";"t2" **\crv{(3,14) & (9, 14)};
   "t1'";"t2'" **\crv{(3,-14) & (9, -14)};
   "t2'";"t2" **\dir{-} ?(.5)*\dir{<};
   (9,0)*{}; (-6,-8)*{\scs i};
 \endxy$} \;\; = \;\; \sum_{f=0}^{-\lambda_i}
   \xy
  (19,4)*{\lambda};
  (0,0)*{\bbe{}};(-2,-8)*{\scs i};
  (12,-2)*{\cbub{\lambda_i-1+f}{i}};
  (0,6)*{\bullet}+(8,1)*{\scs -\lambda_i-f};
 \endxy
\qquad \qquad
  \text{$ \xy 0;/r.18pc/:
  (-12,8)*{\lambda};
   (-3,-8)*{};(3,8)*{} **\crv{(-3,-1) & (3,1)}?(1)*\dir{>};?(0)*\dir{>};
    (3,-8)*{};(-3,8)*{} **\crv{(3,-1) & (-3,1)}?(1)*\dir{>};
  (3,-12)*{\bbsid};
  (3,8)*{\bbsid}; (6,-8)*{\scs i};
  (-9,8)*{}="t1";
  (-3,8)*{}="t2";
  (-9,-8)*{}="t1'";
  (-3,-8)*{}="t2'";
   "t1";"t2" **\crv{(-9,14) & (-3, 14)};
   "t1'";"t2'" **\crv{(-9,-14) & (-3, -14)};
  "t1'";"t1" **\dir{-} ?(.5)*\dir{<};
 \endxy$} \;\; = \;\;
 \sum_{g=0}^{\lambda_i}
   \xy
  (-12,8)*{\lambda};
  (0,0)*{\bbe{}};(2,-8)*{\scs i};
  (-12,-2)*{\ccbub{-\lambda_i-1+g}{i}};
  (0,6)*{\bullet}+(8,-1)*{\scs \lambda_i-g};
 \endxy
\end{eqnarray}
\begin{eqnarray}
 \vcenter{\xy 0;/r.18pc/:
  (-8,0)*{};
  (8,0)*{};
  (-4,10)*{}="t1";
  (4,10)*{}="t2";
  (-4,-10)*{}="b1";
  (4,-10)*{}="b2";(-6,-8)*{\scs i};(6,-8)*{\scs i};
  "t1";"b1" **\dir{-} ?(.5)*\dir{<};
  "t2";"b2" **\dir{-} ?(.5)*\dir{>};
  (10,2)*{\lambda};
  (-10,2)*{\lambda};
  \endxy}
&\quad = \quad&
 \;\;
 \vcenter{   \xy 0;/r.18pc/:
    (-4,-4)*{};(4,4)*{} **\crv{(-4,-1) & (4,1)}?(1)*\dir{>};
    (4,-4)*{};(-4,4)*{} **\crv{(4,-1) & (-4,1)}?(1)*\dir{<};?(0)*\dir{<};
    (-4,4)*{};(4,12)*{} **\crv{(-4,7) & (4,9)};
    (4,4)*{};(-4,12)*{} **\crv{(4,7) & (-4,9)}?(1)*\dir{>};
  (8,8)*{\lambda};(-6,-3)*{\scs i};
     (6.5,-3)*{\scs i};
 \endxy}
  \quad + \quad
   \sum_{f=0}^{\lambda_i-1} \sum_{g=0}^{f}
    \vcenter{\xy 0;/r.18pc/:
    (-10,10)*{\lambda};
    (-8,0)*{};
  (8,0)*{};
  (-4,-15)*{}="b1";
  (4,-15)*{}="b2";
  "b2";"b1" **\crv{(5,-8) & (-5,-8)}; ?(.05)*\dir{<} ?(.93)*\dir{<}
  ?(.8)*\dir{}+(0,-.1)*{\bullet}+(-5,2)*{\scs f-g};
  (-4,15)*{}="t1";
  (4,15)*{}="t2";
  "t2";"t1" **\crv{(5,8) & (-5,8)}; ?(.15)*\dir{>} ?(.95)*\dir{>}
  ?(.4)*\dir{}+(0,-.2)*{\bullet}+(3,-2)*{\scs \;\;\; \lambda_i-1-f};
  (0,0)*{\ccbub{\scs \quad\;\;\;-\lambda_i-1+g}{i}};
  \endxy} \nn
 \\  \; \nn \\
 \vcenter{\xy 0;/r.18pc/:
  (-8,0)*{};(-6,-8)*{\scs i};(6,-8)*{\scs i};
  (8,0)*{};
  (-4,10)*{}="t1";
  (4,10)*{}="t2";
  (-4,-10)*{}="b1";
  (4,-10)*{}="b2";
  "t1";"b1" **\dir{-} ?(.5)*\dir{>};
  "t2";"b2" **\dir{-} ?(.5)*\dir{<};
  (10,2)*{\lambda};
  (-10,2)*{\lambda};
  \endxy}
&\quad = \quad&
 \;\;
   \vcenter{\xy 0;/r.18pc/:
    (-4,-4)*{};(4,4)*{} **\crv{(-4,-1) & (4,1)}?(1)*\dir{<};?(0)*\dir{<};
    (4,-4)*{};(-4,4)*{} **\crv{(4,-1) & (-4,1)}?(1)*\dir{>};
    (-4,4)*{};(4,12)*{} **\crv{(-4,7) & (4,9)}?(1)*\dir{>};
    (4,4)*{};(-4,12)*{} **\crv{(4,7) & (-4,9)};
  (8,8)*{\lambda};(-6,-3)*{\scs i};
     (6,-3)*{\scs i};
 \endxy}
  \quad + \quad
\sum_{f=0}^{-\lambda_i-1} \sum_{g=0}^{f}
    \vcenter{\xy 0;/r.18pc/:
    (-8,0)*{};
  (8,0)*{};
  (-4,-15)*{}="b1";
  (4,-15)*{}="b2";
  "b2";"b1" **\crv{(5,-8) & (-5,-8)}; ?(.1)*\dir{>} ?(.95)*\dir{>}
  ?(.8)*\dir{}+(0,-.1)*{\bullet}+(-5,2)*{\scs f-g};
  (-4,15)*{}="t1";
  (4,15)*{}="t2";
  "t2";"t1" **\crv{(5,8) & (-5,8)}; ?(.15)*\dir{<} ?(.97)*\dir{<}
  ?(.4)*\dir{}+(0,-.2)*{\bullet}+(3,-2)*{\scs \;\;-\lambda_i-1-f};
  (0,0)*{\cbub{\scs \quad\; \lambda_i-1+g}{i}};
  (-10,10)*{\lambda};
  \endxy} \label{eq_ident_decomp}
\end{eqnarray}
for all $\lambda\in X$.  Notice that for some values of $\lambda$ the dotted
bubbles appearing above have negative labels. A composite of $\Uup_{i,\lambda}$
or $\Udown_{i,\lambda}$ with itself a negative number of times does not make
sense. These dotted bubbles with negative labels, called {\em fake bubbles}, are
formal symbols inductively defined by the equation
\begin{center}
\begin{eqnarray}
 \makebox[0pt]{ $
\left( \xy 0;/r.15pc/:
 (0,0)*{\ccbub{-\lambda_i-1}{i}};
  (4,8)*{\lambda};
 \endxy
 +
 \xy 0;/r.15pc/:
 (0,0)*{\ccbub{-\lambda_i-1+1}{i}};
  (4,8)*{\lambda};
 \endxy t
 + \cdots +
 \xy 0;/r.15pc/:
 (0,0)*{\ccbub{-\lambda_i-1+\alpha}{i}};
  (4,8)*{\lambda};
 \endxy t^{\alpha}
 + \cdots
\right)
%
\left( \xy 0;/r.15pc/:
 (0,0)*{\cbub{\lambda_i-1}{i}};
  (4,8)*{\lambda};
 \endxy
 + \cdots +
 \xy 0;/r.15pc/:
 (0,0)*{\cbub{\lambda_i-1+\alpha}{i}};
 (4,8)*{\lambda};
 \endxy t^{\alpha}
 + \cdots
\right) =1 .$ } \nn \\ \label{eq_infinite_Grass}
\end{eqnarray}
\end{center}
and the additional condition
\[
\xy 0;/r.18pc/:
 (0,0)*{\cbub{-1}{i}};
  (4,8)*{\lambda};
 \endxy
 \quad = \quad
  \xy 0;/r.18pc/:
 (0,0)*{\ccbub{-1}{i}};
  (4,8)*{\lambda};
 \endxy
  \quad = \quad 1 \quad \text{if $\lambda_i =0$.}
\]
Although the labels are negative for fake bubbles, one can check that the overall
degree of each fake bubble is still positive, so that these fake bubbles do not
violate the positivity of dotted bubble axiom. The above equation, called the
infinite Grassmannian relation, remains valid even in high degree when most of
the bubbles involved are not fake bubbles.  See \cite{Lau1} for more details.
\item NilHecke relations:
 \begin{equation}
  \vcenter{\xy 0;/r.18pc/:
    (-4,-4)*{};(4,4)*{} **\crv{(-4,-1) & (4,1)}?(1)*\dir{>};
    (4,-4)*{};(-4,4)*{} **\crv{(4,-1) & (-4,1)}?(1)*\dir{>};
    (-4,4)*{};(4,12)*{} **\crv{(-4,7) & (4,9)}?(1)*\dir{>};
    (4,4)*{};(-4,12)*{} **\crv{(4,7) & (-4,9)}?(1)*\dir{>};
  (8,8)*{\lambda};(-5,-3)*{\scs i};
     (5.1,-3)*{\scs i};
 \endxy}
 =0, \qquad \quad
 \vcenter{
 \xy 0;/r.18pc/:
    (-4,-4)*{};(4,4)*{} **\crv{(-4,-1) & (4,1)}?(1)*\dir{>};
    (4,-4)*{};(-4,4)*{} **\crv{(4,-1) & (-4,1)}?(1)*\dir{>};
    (4,4)*{};(12,12)*{} **\crv{(4,7) & (12,9)}?(1)*\dir{>};
    (12,4)*{};(4,12)*{} **\crv{(12,7) & (4,9)}?(1)*\dir{>};
    (-4,12)*{};(4,20)*{} **\crv{(-4,15) & (4,17)}?(1)*\dir{>};
    (4,12)*{};(-4,20)*{} **\crv{(4,15) & (-4,17)}?(1)*\dir{>};
    (-4,4)*{}; (-4,12) **\dir{-};
    (12,-4)*{}; (12,4) **\dir{-};
    (12,12)*{}; (12,20) **\dir{-}; (-5.5,-3)*{\scs i};
     (5.5,-3)*{\scs i};(14,-3)*{\scs i};
  (18,8)*{\lambda};
\endxy}
 \;\; =\;\;
 \vcenter{
 \xy 0;/r.18pc/:
    (4,-4)*{};(-4,4)*{} **\crv{(4,-1) & (-4,1)}?(1)*\dir{>};
    (-4,-4)*{};(4,4)*{} **\crv{(-4,-1) & (4,1)}?(1)*\dir{>};
    (-4,4)*{};(-12,12)*{} **\crv{(-4,7) & (-12,9)}?(1)*\dir{>};
    (-12,4)*{};(-4,12)*{} **\crv{(-12,7) & (-4,9)}?(1)*\dir{>};
    (4,12)*{};(-4,20)*{} **\crv{(4,15) & (-4,17)}?(1)*\dir{>};
    (-4,12)*{};(4,20)*{} **\crv{(-4,15) & (4,17)}?(1)*\dir{>};
    (4,4)*{}; (4,12) **\dir{-};
    (-12,-4)*{}; (-12,4) **\dir{-};
    (-12,12)*{}; (-12,20) **\dir{-};(-5.5,-3)*{\scs i};
     (5.5,-3)*{\scs i};(-14,-3)*{\scs i};
  (10,8)*{\lambda};
\endxy} \label{eq_nil_rels}
  \end{equation}
\begin{eqnarray}
  \xy
  (4,4);(4,-4) **\dir{-}?(0)*\dir{<}+(2.3,0)*{};
  (-4,4);(-4,-4) **\dir{-}?(0)*\dir{<}+(2.3,0)*{};
  (9,2)*{\lambda};     (-5,-3)*{\scs i};
     (5.1,-3)*{\scs i};
 \endxy
 \quad =
\xy
  (0,0)*{\xybox{
    (-4,-4)*{};(4,4)*{} **\crv{(-4,-1) & (4,1)}?(1)*\dir{>}?(.25)*{\bullet};
    (4,-4)*{};(-4,4)*{} **\crv{(4,-1) & (-4,1)}?(1)*\dir{>};
    (-5,-3)*{\scs i};
     (5.1,-3)*{\scs i};
     (8,1)*{ \lambda};
     (-10,0)*{};(10,0)*{};
     }};
  \endxy
 \;\; +
 \xy
  (0,0)*{\xybox{
    (-4,-4)*{};(4,4)*{} **\crv{(-4,-1) & (4,1)}?(1)*\dir{>}?(.75)*{\bullet};
    (4,-4)*{};(-4,4)*{} **\crv{(4,-1) & (-4,1)}?(1)*\dir{>};
    (-5,-3)*{\scs i};
     (5.1,-3)*{\scs i};
     (8,1)*{ \lambda};
     (-10,0)*{};(10,0)*{};
     }};
  \endxy
 \;\; =
\xy
  (0,0)*{\xybox{
    (-4,-4)*{};(4,4)*{} **\crv{(-4,-1) & (4,1)}?(1)*\dir{>};
    (4,-4)*{};(-4,4)*{} **\crv{(4,-1) & (-4,1)}?(1)*\dir{>}?(.75)*{\bullet};
    (-5,-3)*{\scs i};
     (5.1,-3)*{\scs i};
     (8,1)*{ \lambda};
     (-10,0)*{};(10,0)*{};
     }};
  \endxy
 \;\; +
  \xy
  (0,0)*{\xybox{
    (-4,-4)*{};(4,4)*{} **\crv{(-4,-1) & (4,1)}?(1)*\dir{>} ;
    (4,-4)*{};(-4,4)*{} **\crv{(4,-1) & (-4,1)}?(1)*\dir{>}?(.25)*{\bullet};
    (-5,-3)*{\scs i};
     (5.1,-3)*{\scs i};
     (8,1)*{ \lambda};
     (-10,0)*{};(10,0)*{};
     }};
  \endxy \nn \\ \label{eq_nil_dotslide}
\end{eqnarray}
We will also include \eqref{eq_cyclic_cross-gen} for $i =j$ as an
$\mf{sl}_2$-relation.
\end{enumerate}

  \item All 2-morphisms are cyclic\footnote{See \cite{Lau1} and the references therein for
  the definition of a cyclic 2-morphism with respect to a biadjoint structure.} with respect to the above biadjoint
   structure.  This is ensured by the relations \eqref{eq_cyclic_dot}, and the
   relations
\begin{equation} \label{eq_cyclic_cross-gen}
\xy 0;/r.19pc/:
  (0,0)*{\xybox{
    (-4,-4)*{};(4,4)*{} **\crv{(-4,-1) & (4,1)}?(1)*\dir{>};
    (4,-4)*{};(-4,4)*{} **\crv{(4,-1) & (-4,1)};
     (4,4)*{};(-18,4)*{} **\crv{(4,16) & (-18,16)} ?(1)*\dir{>};
     (-4,-4)*{};(18,-4)*{} **\crv{(-4,-16) & (18,-16)} ?(1)*\dir{<}?(0)*\dir{<};
     (18,-4);(18,12) **\dir{-};(12,-4);(12,12) **\dir{-};
     (-18,4);(-18,-12) **\dir{-};(-12,4);(-12,-12) **\dir{-};
     (8,1)*{ \lambda};
     (-10,0)*{};(10,0)*{};
      (4,-4)*{};(12,-4)*{} **\crv{(4,-10) & (12,-10)}?(1)*\dir{<}?(0)*\dir{<};
      (-4,4)*{};(-12,4)*{} **\crv{(-4,10) & (-12,10)}?(1)*\dir{>}?(0)*\dir{>};
      (20,11)*{\scs j};(10,11)*{\scs i};
      (-20,-11)*{\scs j};(-10,-11)*{\scs i};
     }};
  \endxy
\quad =  \quad \xy
  (0,0)*{\xybox{
    (-4,-4)*{};(4,4)*{} **\crv{(-4,-1) & (4,1)}?(0)*\dir{<} ;
    (4,-4)*{};(-4,4)*{} **\crv{(4,-1) & (-4,1)}?(0)*\dir{<};
    (-5,3)*{\scs i};
     (5.1,3)*{\scs j};
     (-8,0)*{ \lambda};
     (-12,0)*{};(12,0)*{};
     }};
  \endxy \quad =  \quad
 \xy 0;/r.19pc/:
  (0,0)*{\xybox{
    (4,-4)*{};(-4,4)*{} **\crv{(4,-1) & (-4,1)}?(1)*\dir{>};
    (-4,-4)*{};(4,4)*{} **\crv{(-4,-1) & (4,1)};
     (-4,4)*{};(18,4)*{} **\crv{(-4,16) & (18,16)} ?(1)*\dir{>};
     (4,-4)*{};(-18,-4)*{} **\crv{(4,-16) & (-18,-16)} ?(1)*\dir{<}?(0)*\dir{<};
     (-18,-4);(-18,12) **\dir{-};(-12,-4);(-12,12) **\dir{-};
     (18,4);(18,-12) **\dir{-};(12,4);(12,-12) **\dir{-};
     (8,1)*{ \lambda};
     (-10,0)*{};(10,0)*{};
     (-4,-4)*{};(-12,-4)*{} **\crv{(-4,-10) & (-12,-10)}?(1)*\dir{<}?(0)*\dir{<};
      (4,4)*{};(12,4)*{} **\crv{(4,10) & (12,10)}?(1)*\dir{>}?(0)*\dir{>};
      (-20,11)*{\scs i};(-10,11)*{\scs j};
      (20,-11)*{\scs i};(10,-11)*{\scs j};
     }};
  \endxy
\end{equation}
The cyclic condition on 2-morphisms expressed by \eqref{eq_cyclic_dot} and
\eqref{eq_cyclic_cross-gen} ensures that diagrams related by isotopy represent
the same 2-morphism in $\Ucat$.

It will be convenient to introduce degree zero 2-morphisms:
\begin{equation} \label{eq_crossl-gen}
  \xy
  (0,0)*{\xybox{
    (-4,-4)*{};(4,4)*{} **\crv{(-4,-1) & (4,1)}?(1)*\dir{>} ;
    (4,-4)*{};(-4,4)*{} **\crv{(4,-1) & (-4,1)}?(0)*\dir{<};
    (-5,-3)*{\scs i};
     (-5,3)*{\scs j};
     (8,2)*{ \lambda};
     (-12,0)*{};(12,0)*{};
     }};
  \endxy
:=
 \xy 0;/r.19pc/:
  (0,0)*{\xybox{
    (4,-4)*{};(-4,4)*{} **\crv{(4,-1) & (-4,1)}?(1)*\dir{>};
    (-4,-4)*{};(4,4)*{} **\crv{(-4,-1) & (4,1)};
     (-4,4);(-4,12) **\dir{-};
     (-12,-4);(-12,12) **\dir{-};
     (4,-4);(4,-12) **\dir{-};(12,4);(12,-12) **\dir{-};
     (16,1)*{\lambda};
     (-10,0)*{};(10,0)*{};
     (-4,-4)*{};(-12,-4)*{} **\crv{(-4,-10) & (-12,-10)}?(1)*\dir{<}?(0)*\dir{<};
      (4,4)*{};(12,4)*{} **\crv{(4,10) & (12,10)}?(1)*\dir{>}?(0)*\dir{>};
      (-14,11)*{\scs j};(-2,11)*{\scs i};
      (14,-11)*{\scs j};(2,-11)*{\scs i};
     }};
  \endxy
  \quad = \quad
  \xy 0;/r.19pc/:
  (0,0)*{\xybox{
    (-4,-4)*{};(4,4)*{} **\crv{(-4,-1) & (4,1)}?(1)*\dir{<};
    (4,-4)*{};(-4,4)*{} **\crv{(4,-1) & (-4,1)};
     (4,4);(4,12) **\dir{-};
     (12,-4);(12,12) **\dir{-};
     (-4,-4);(-4,-12) **\dir{-};(-12,4);(-12,-12) **\dir{-};
     (16,1)*{\lambda};
     (10,0)*{};(-10,0)*{};
     (4,-4)*{};(12,-4)*{} **\crv{(4,-10) & (12,-10)}?(1)*\dir{>}?(0)*\dir{>};
      (-4,4)*{};(-12,4)*{} **\crv{(-4,10) & (-12,10)}?(1)*\dir{<}?(0)*\dir{<};
      (14,11)*{\scs i};(2,11)*{\scs j};
      (-14,-11)*{\scs i};(-2,-11)*{\scs j};
     }};
  \endxy
\end{equation}
\begin{equation} \label{eq_crossr-gen}
  \xy
  (0,0)*{\xybox{
    (-4,-4)*{};(4,4)*{} **\crv{(-4,-1) & (4,1)}?(0)*\dir{<} ;
    (4,-4)*{};(-4,4)*{} **\crv{(4,-1) & (-4,1)}?(1)*\dir{>};
    (5.1,-3)*{\scs i};
     (5.1,3)*{\scs j};
     (-8,2)*{ \lambda};
     (-12,0)*{};(12,0)*{};
     }};
  \endxy
:=
 \xy 0;/r.19pc/:
  (0,0)*{\xybox{
    (-4,-4)*{};(4,4)*{} **\crv{(-4,-1) & (4,1)}?(1)*\dir{>};
    (4,-4)*{};(-4,4)*{} **\crv{(4,-1) & (-4,1)};
     (4,4);(4,12) **\dir{-};
     (12,-4);(12,12) **\dir{-};
     (-4,-4);(-4,-12) **\dir{-};(-12,4);(-12,-12) **\dir{-};
     (-16,1)*{\lambda};
     (10,0)*{};(-10,0)*{};
     (4,-4)*{};(12,-4)*{} **\crv{(4,-10) & (12,-10)}?(1)*\dir{<}?(0)*\dir{<};
      (-4,4)*{};(-12,4)*{} **\crv{(-4,10) & (-12,10)}?(1)*\dir{>}?(0)*\dir{>};
      (14,11)*{\scs j};(2,11)*{\scs i};
      (-14,-11)*{\scs j};(-2,-11)*{\scs i};
     }};
  \endxy
  \quad = \quad
  \xy 0;/r.19pc/:
  (0,0)*{\xybox{
    (4,-4)*{};(-4,4)*{} **\crv{(4,-1) & (-4,1)}?(1)*\dir{<};
    (-4,-4)*{};(4,4)*{} **\crv{(-4,-1) & (4,1)};
     (-4,4);(-4,12) **\dir{-};
     (-12,-4);(-12,12) **\dir{-};
     (4,-4);(4,-12) **\dir{-};(12,4);(12,-12) **\dir{-};
     (-16,1)*{\lambda};
     (-10,0)*{};(10,0)*{};
     (-4,-4)*{};(-12,-4)*{} **\crv{(-4,-10) & (-12,-10)}?(1)*\dir{>}?(0)*\dir{>};
      (4,4)*{};(12,4)*{} **\crv{(4,10) & (12,10)}?(1)*\dir{<}?(0)*\dir{<};
      (-14,11)*{\scs i};(-2,11)*{\scs j};
      (14,-11)*{\scs i};(2,-11)*{\scs j};
     }};
  \endxy
\end{equation}
where the second equality in \eqref{eq_crossl-gen} and \eqref{eq_crossr-gen}
follow from \eqref{eq_cyclic_cross-gen}.

\item For $i \neq j$
\begin{equation} \label{eq_downup_ij-gen}
 \vcenter{   \xy 0;/r.18pc/:
    (-4,-4)*{};(4,4)*{} **\crv{(-4,-1) & (4,1)}?(1)*\dir{>};
    (4,-4)*{};(-4,4)*{} **\crv{(4,-1) & (-4,1)}?(1)*\dir{<};?(0)*\dir{<};
    (-4,4)*{};(4,12)*{} **\crv{(-4,7) & (4,9)};
    (4,4)*{};(-4,12)*{} **\crv{(4,7) & (-4,9)}?(1)*\dir{>};
  (8,8)*{\lambda};(-6,-3)*{\scs i};
     (6,-3)*{\scs j};
 \endxy}
 \;\;= \;\;
\xy 0;/r.18pc/:
  (3,9);(3,-9) **\dir{-}?(.55)*\dir{>}+(2.3,0)*{};
  (-3,9);(-3,-9) **\dir{-}?(.5)*\dir{<}+(2.3,0)*{};
  (8,2)*{\lambda};(-5,-6)*{\scs i};     (5.1,-6)*{\scs j};
 \endxy
 \qquad
    \vcenter{\xy 0;/r.18pc/:
    (-4,-4)*{};(4,4)*{} **\crv{(-4,-1) & (4,1)}?(1)*\dir{<};?(0)*\dir{<};
    (4,-4)*{};(-4,4)*{} **\crv{(4,-1) & (-4,1)}?(1)*\dir{>};
    (-4,4)*{};(4,12)*{} **\crv{(-4,7) & (4,9)}?(1)*\dir{>};
    (4,4)*{};(-4,12)*{} **\crv{(4,7) & (-4,9)};
  (8,8)*{\lambda};(-6,-3)*{\scs i};
     (6,-3)*{\scs j};
 \endxy}
 \;\;=\;\;
\xy 0;/r.18pc/:
  (3,9);(3,-9) **\dir{-}?(.5)*\dir{<}+(2.3,0)*{};
  (-3,9);(-3,-9) **\dir{-}?(.55)*\dir{>}+(2.3,0)*{};
  (8,2)*{\lambda};(-5,-6)*{\scs i};     (5.1,-6)*{\scs j};
 \endxy
\end{equation}

\item The $R(\nu)$-relations:
\begin{enumerate}[i)]
\item For $i \neq j$
\begin{eqnarray}
  \vcenter{\xy 0;/r.18pc/:
    (-4,-4)*{};(4,4)*{} **\crv{(-4,-1) & (4,1)}?(1)*\dir{>};
    (4,-4)*{};(-4,4)*{} **\crv{(4,-1) & (-4,1)}?(1)*\dir{>};
    (-4,4)*{};(4,12)*{} **\crv{(-4,7) & (4,9)}?(1)*\dir{>};
    (4,4)*{};(-4,12)*{} **\crv{(4,7) & (-4,9)}?(1)*\dir{>};
  (8,8)*{\lambda};(-5,-3)*{\scs i};
     (5.1,-3)*{\scs j};
 \endxy}
 \qquad = \qquad
 \left\{
 \begin{array}{ccc}
     \xy 0;/r.18pc/:
  (3,9);(3,-9) **\dir{-}?(.5)*\dir{<}+(2.3,0)*{};
  (-3,9);(-3,-9) **\dir{-}?(.5)*\dir{<}+(2.3,0)*{};
  (8,2)*{\lambda};(-5,-6)*{\scs i};     (5.1,-6)*{\scs j};
 \endxy &  &  \text{if $i \cdot j=0$,}\\ \\
  \vcenter{\xy 0;/r.18pc/:
  (3,9);(3,-9) **\dir{-}?(.5)*\dir{<}+(2.3,0)*{};
  (-3,9);(-3,-9) **\dir{-}?(.5)*\dir{<}+(2.3,0)*{};
  (8,2)*{\lambda}; (-3,4)*{\bullet};
  (-5,-6)*{\scs i};     (5.1,-6)*{\scs j};
 \endxy} \quad
 + \quad
 \vcenter{\xy 0;/r.18pc/:
  (3,9);(3,-9) **\dir{-}?(.5)*\dir{<}+(2.3,0)*{};
  (-3,9);(-3,-9) **\dir{-}?(.5)*\dir{<}+(2.3,0)*{};
  (12,2)*{\lambda}; (3,4)*{\bullet};
  (-5,-6)*{\scs i};     (5.1,-6)*{\scs j};
 \endxy}
   &  & \text{if $i \cdot j =-1$.}
 \end{array}
 \right. \nn \\\label{eq_r2_ij-gen}
\end{eqnarray}

\begin{eqnarray} \label{eq_dot_slide_ij-gen}
\xy
  (0,0)*{\xybox{
    (-4,-4)*{};(4,4)*{} **\crv{(-4,-1) & (4,1)}?(1)*\dir{>}?(.75)*{\bullet};
    (4,-4)*{};(-4,4)*{} **\crv{(4,-1) & (-4,1)}?(1)*\dir{>};
    (-5,-3)*{\scs i};
     (5.1,-3)*{\scs j};
     (8,1)*{ \lambda};
     (-10,0)*{};(10,0)*{};
     }};
  \endxy
 \;\; =
\xy
  (0,0)*{\xybox{
    (-4,-4)*{};(4,4)*{} **\crv{(-4,-1) & (4,1)}?(1)*\dir{>}?(.25)*{\bullet};
    (4,-4)*{};(-4,4)*{} **\crv{(4,-1) & (-4,1)}?(1)*\dir{>};
    (-5,-3)*{\scs i};
     (5.1,-3)*{\scs j};
     (8,1)*{ \lambda};
     (-10,0)*{};(10,0)*{};
     }};
  \endxy
\qquad  \xy
  (0,0)*{\xybox{
    (-4,-4)*{};(4,4)*{} **\crv{(-4,-1) & (4,1)}?(1)*\dir{>};
    (4,-4)*{};(-4,4)*{} **\crv{(4,-1) & (-4,1)}?(1)*\dir{>}?(.75)*{\bullet};
    (-5,-3)*{\scs i};
     (5.1,-3)*{\scs j};
     (8,1)*{ \lambda};
     (-10,0)*{};(10,0)*{};
     }};
  \endxy
\;\;  =
  \xy
  (0,0)*{\xybox{
    (-4,-4)*{};(4,4)*{} **\crv{(-4,-1) & (4,1)}?(1)*\dir{>} ;
    (4,-4)*{};(-4,4)*{} **\crv{(4,-1) & (-4,1)}?(1)*\dir{>}?(.25)*{\bullet};
    (-5,-3)*{\scs i};
     (5.1,-3)*{\scs j};
     (8,1)*{ \lambda};
     (-10,0)*{};(12,0)*{};
     }};
  \endxy
\end{eqnarray}

\item Unless $i = k$ and $i \cdot j=-1$
\begin{equation}
 \vcenter{
 \xy 0;/r.18pc/:
    (-4,-4)*{};(4,4)*{} **\crv{(-4,-1) & (4,1)}?(1)*\dir{>};
    (4,-4)*{};(-4,4)*{} **\crv{(4,-1) & (-4,1)}?(1)*\dir{>};
    (4,4)*{};(12,12)*{} **\crv{(4,7) & (12,9)}?(1)*\dir{>};
    (12,4)*{};(4,12)*{} **\crv{(12,7) & (4,9)}?(1)*\dir{>};
    (-4,12)*{};(4,20)*{} **\crv{(-4,15) & (4,17)}?(1)*\dir{>};
    (4,12)*{};(-4,20)*{} **\crv{(4,15) & (-4,17)}?(1)*\dir{>};
    (-4,4)*{}; (-4,12) **\dir{-};
    (12,-4)*{}; (12,4) **\dir{-};
    (12,12)*{}; (12,20) **\dir{-};
  (18,8)*{\lambda};
  (-6,-3)*{\scs i};
  (6,-3)*{\scs j};
  (15,-3)*{\scs k};
\endxy}
 \;\; =\;\;
 \vcenter{
 \xy 0;/r.18pc/:
    (4,-4)*{};(-4,4)*{} **\crv{(4,-1) & (-4,1)}?(1)*\dir{>};
    (-4,-4)*{};(4,4)*{} **\crv{(-4,-1) & (4,1)}?(1)*\dir{>};
    (-4,4)*{};(-12,12)*{} **\crv{(-4,7) & (-12,9)}?(1)*\dir{>};
    (-12,4)*{};(-4,12)*{} **\crv{(-12,7) & (-4,9)}?(1)*\dir{>};
    (4,12)*{};(-4,20)*{} **\crv{(4,15) & (-4,17)}?(1)*\dir{>};
    (-4,12)*{};(4,20)*{} **\crv{(-4,15) & (4,17)}?(1)*\dir{>};
    (4,4)*{}; (4,12) **\dir{-};
    (-12,-4)*{}; (-12,4) **\dir{-};
    (-12,12)*{}; (-12,20) **\dir{-};
  (10,8)*{\lambda};
  (7,-3)*{\scs k};
  (-6,-3)*{\scs j};
  (-14,-3)*{\scs i};
\endxy} \label{eq_r3_easy-gen}
\end{equation}

For $i \cdot j =-1$
\begin{equation}
 \vcenter{
 \xy 0;/r.18pc/:
    (-4,-4)*{};(4,4)*{} **\crv{(-4,-1) & (4,1)}?(1)*\dir{>};
    (4,-4)*{};(-4,4)*{} **\crv{(4,-1) & (-4,1)}?(1)*\dir{>};
    (4,4)*{};(12,12)*{} **\crv{(4,7) & (12,9)}?(1)*\dir{>};
    (12,4)*{};(4,12)*{} **\crv{(12,7) & (4,9)}?(1)*\dir{>};
    (-4,12)*{};(4,20)*{} **\crv{(-4,15) & (4,17)}?(1)*\dir{>};
    (4,12)*{};(-4,20)*{} **\crv{(4,15) & (-4,17)}?(1)*\dir{>};
    (-4,4)*{}; (-4,12) **\dir{-};
    (12,-4)*{}; (12,4) **\dir{-};
    (12,12)*{}; (12,20) **\dir{-};
  (18,8)*{\lambda};
  (-6,-3)*{\scs i};
  (6,-3)*{\scs j};
  (14,-3)*{\scs i};
\endxy}
\quad + \quad
 \vcenter{
 \xy 0;/r.18pc/:
    (4,-4)*{};(-4,4)*{} **\crv{(4,-1) & (-4,1)}?(1)*\dir{>};
    (-4,-4)*{};(4,4)*{} **\crv{(-4,-1) & (4,1)}?(1)*\dir{>};
    (-4,4)*{};(-12,12)*{} **\crv{(-4,7) & (-12,9)}?(1)*\dir{>};
    (-12,4)*{};(-4,12)*{} **\crv{(-12,7) & (-4,9)}?(1)*\dir{>};
    (4,12)*{};(-4,20)*{} **\crv{(4,15) & (-4,17)}?(1)*\dir{>};
    (-4,12)*{};(4,20)*{} **\crv{(-4,15) & (4,17)}?(1)*\dir{>};
    (4,4)*{}; (4,12) **\dir{-};
    (-12,-4)*{}; (-12,4) **\dir{-};
    (-12,12)*{}; (-12,20) **\dir{-};
  (10,8)*{\lambda};
  (6,-3)*{\scs i};
  (-6,-3)*{\scs j};
  (-14,-3)*{\scs i};
\endxy}
 \;\; =\;\;\;\;
\xy 0;/r.18pc/:
  (4,12);(4,-12) **\dir{-}?(.5)*\dir{<};
  (-4,12);(-4,-12) **\dir{-}?(.5)*\dir{<}?(.25)*\dir{}+(0,0)*{}+(-3,0)*{\scs };
  (12,12);(12,-12) **\dir{-}?(.5)*\dir{<}?(.25)*\dir{}+(0,0)*{}+(10,0)*{\scs};
  (22,-2)*{\lambda}; (-6,-9)*{\scs i};     (6.1,-9)*{\scs j};
  (14,-9)*{\scs i};
 \endxy
 \label{eq_r3_hard-gen}
\end{equation}
\end{enumerate}
For example, for any shift $t$ there are 2-morphisms
\begin{eqnarray}
  \xy
 (0,7);(0,-7); **\dir{-} ?(.75)*\dir{>}+(2.3,0)*{\scriptstyle{}};
 (0.1,-2)*{\txt\large{$\bullet$}};
 (6,4)*{ \lambda};
 (-10,0)*{};(10,0)*{};(0,-10)*{i };(0,10)*{i};
 \endxy
 \maps  \cal{E}_{+i}\onel\{t\} \To \cal{E}_{+i}\onel\{t-2\}\quad
\xy
  (0,0)*{\xybox{
    (-4,-6)*{};(4,6)*{} **\crv{(-4,-1) & (4,1)}?(1)*\dir{>} ;
    (4,-6)*{};(-4,6)*{} **\crv{(4,-1) & (-4,1)}?(1)*\dir{>};
    (-4,-9)*{ i};
     (4,-9)*{ j};
     (8,1)*{ \lambda};
     (-12,0)*{};(12,0)*{};
     }};
  \endxy
  \maps \cal{E}_{+i+j}\onel\{t\} \To \cal{E}_{+j+i}\onel\{t-i\cdot j\} \nn \\
  \xy
    (0,-3)*{\bbpef{i}};
    (8,-5)*{ \lambda};
    (-12,0)*{};(12,0)*{};
    \endxy \maps
    \onel\{t\} \To \cal{E}_{-i+i}\onel\{t-c_{+i,\lambda}\} \quad \xy
    (0,0)*{\bbcfe{i}};
    (8,5)*{ \lambda};
    (-12,0)*{};(12,0)*{};
    \endxy \maps
    \cal{E}_{-i+i}\onel\{t\} \To \onel\{t-c_{-i,\lambda}\} \nn
\end{eqnarray}
in $\Ucat$,  and the diagrammatic relation
\[
\vcenter{
 \xy 0;/r.18pc/:
    (-4,-4)*{};(4,4)*{} **\crv{(-4,-1) & (4,1)}?(1)*\dir{>};
    (4,-4)*{};(-4,4)*{} **\crv{(4,-1) & (-4,1)}?(1)*\dir{>};
    (4,4)*{};(12,12)*{} **\crv{(4,7) & (12,9)}?(1)*\dir{>};
    (12,4)*{};(4,12)*{} **\crv{(12,7) & (4,9)}?(1)*\dir{>};
    (-4,12)*{};(4,20)*{} **\crv{(-4,15) & (4,17)}?(1)*\dir{>};
    (4,12)*{};(-4,20)*{} **\crv{(4,15) & (-4,17)}?(1)*\dir{>};
    (-4,4)*{}; (-4,12) **\dir{-};
    (12,-4)*{}; (12,4) **\dir{-};
    (12,12)*{}; (12,20) **\dir{-}; (-5.5,-3)*{\scs i};
     (5.5,-3)*{\scs i};(14,-3)*{\scs i};
  (18,8)*{\lambda};
\endxy}
 \;\; =\;\;
 \vcenter{
 \xy 0;/r.18pc/:
    (4,-4)*{};(-4,4)*{} **\crv{(4,-1) & (-4,1)}?(1)*\dir{>};
    (-4,-4)*{};(4,4)*{} **\crv{(-4,-1) & (4,1)}?(1)*\dir{>};
    (-4,4)*{};(-12,12)*{} **\crv{(-4,7) & (-12,9)}?(1)*\dir{>};
    (-12,4)*{};(-4,12)*{} **\crv{(-12,7) & (-4,9)}?(1)*\dir{>};
    (4,12)*{};(-4,20)*{} **\crv{(4,15) & (-4,17)}?(1)*\dir{>};
    (-4,12)*{};(4,20)*{} **\crv{(-4,15) & (4,17)}?(1)*\dir{>};
    (4,4)*{}; (4,12) **\dir{-};
    (-12,-4)*{}; (-12,4) **\dir{-};
    (-12,12)*{}; (-12,20) **\dir{-};(-5.5,-3)*{\scs i};
     (5.5,-3)*{\scs i};(-14,-3)*{\scs i};
  (10,8)*{\lambda};
\endxy}
\]
gives rise to relations in
$\Ucat\big(\cal{E}_{iii}\onel\{t\},\cal{E}_{iii}\onel\{t+3i\cdot i\}\big)$ for
all $t\in \Z$.

\item the additive $\Z$-linear composition functor $\Ucat(\lambda,\lambda')
 \times \Ucat(\lambda',\lambda'') \to \Ucat(\lambda,\lambda'')$ is given on
 1-morphisms of $\Ucat$ by
\begin{equation}
  \cal{E}_{\jj}\mathbf{1}_{\lambda'}\{t'\} \times \cal{E}_{\ii}\onel\{t\} \mapsto
  \cal{E}_{\jj\ii}\onel\{t+t'\}
\end{equation}
for $\ii_X=\lambda-\lambda'$, and on 2-morphisms of $\Ucat$ by juxtaposition of
diagrams
\[
\left(\;\;\vcenter{\xy 0;/r.16pc/:
 (-4,-15)*{}; (-20,25) **\crv{(-3,-6) & (-20,4)}?(0)*\dir{<}?(.6)*\dir{}+(0,0)*{\bullet};
 (-12,-15)*{}; (-4,25) **\crv{(-12,-6) & (-4,0)}?(0)*\dir{<}?(.6)*\dir{}+(.2,0)*{\bullet};
 ?(0)*\dir{<}?(.75)*\dir{}+(.2,0)*{\bullet};?(0)*\dir{<}?(.9)*\dir{}+(0,0)*{\bullet};
 (-28,25)*{}; (-12,25) **\crv{(-28,10) & (-12,10)}?(0)*\dir{<};
  ?(.2)*\dir{}+(0,0)*{\bullet}?(.35)*\dir{}+(0,0)*{\bullet};
 (-36,-15)*{}; (-36,25) **\crv{(-34,-6) & (-35,4)}?(1)*\dir{>};
 (-28,-15)*{}; (-42,25) **\crv{(-28,-6) & (-42,4)}?(1)*\dir{>};
 (-42,-15)*{}; (-20,-15) **\crv{(-42,-5) & (-20,-5)}?(1)*\dir{>};
 (6,10)*{\cbub{}{}};
 (-23,0)*{\cbub{}{}};
 (8,-4)*{\lambda'};(-44,-4)*{\lambda''};
 \endxy}\;\;\right) \;\; \times \;\;
\left(\;\;\vcenter{ \xy 0;/r.18pc/: (-14,8)*{\xybox{
 (0,-10)*{}; (-16,10)*{} **\crv{(0,-6) & (-16,6)}?(.5)*\dir{};
 (-16,-10)*{}; (-8,10)*{} **\crv{(-16,-6) & (-8,6)}?(1)*\dir{}+(.1,0)*{\bullet};
  (-8,-10)*{}; (0,10)*{} **\crv{(-8,-6) & (-0,6)}?(.6)*\dir{}+(.2,0)*{\bullet}?
  (1)*\dir{}+(.1,0)*{\bullet};
  (0,10)*{}; (-16,30)*{} **\crv{(0,14) & (-16,26)}?(1)*\dir{>};
 (-16,10)*{}; (-8,30)*{} **\crv{(-16,14) & (-8,26)}?(1)*\dir{>};
  (-8,10)*{}; (0,30)*{} **\crv{(-8,14) & (-0,26)}?(1)*\dir{>}?(.6)*\dir{}+(.25,0)*{\bullet};
   }};
 (-2,-4)*{\lambda}; (-26,-4)*{\lambda'};
 \endxy} \;\;\right)
 \;\;\mapsto \;\;
\vcenter{\xy 0;/r.16pc/:
 (-4,-15)*{}; (-20,25) **\crv{(-3,-6) & (-20,4)}?(0)*\dir{<}?(.6)*\dir{}+(0,0)*{\bullet};
 (-12,-15)*{}; (-4,25) **\crv{(-12,-6) & (-4,0)}?(0)*\dir{<}?(.6)*\dir{}+(.2,0)*{\bullet};
 ?(0)*\dir{<}?(.75)*\dir{}+(.2,0)*{\bullet};?(0)*\dir{<}?(.9)*\dir{}+(0,0)*{\bullet};
 (-28,25)*{}; (-12,25) **\crv{(-28,10) & (-12,10)}?(0)*\dir{<};
  ?(.2)*\dir{}+(0,0)*{\bullet}?(.35)*\dir{}+(0,0)*{\bullet};
 (-36,-15)*{}; (-36,25) **\crv{(-34,-6) & (-35,4)}?(1)*\dir{>};
 (-28,-15)*{}; (-42,25) **\crv{(-28,-6) & (-42,4)}?(1)*\dir{>};
 (-42,-15)*{}; (-20,-15) **\crv{(-42,-5) & (-20,-5)}?(1)*\dir{>};
 (6,10)*{\cbub{}{}};
 (-23,0)*{\cbub{}{}};
 \endxy}
 \vcenter{ \xy 0;/r.16pc/: (-14,8)*{\xybox{
 (0,-10)*{}; (-16,10)*{} **\crv{(0,-6) & (-16,6)}?(.5)*\dir{};
 (-16,-10)*{}; (-8,10)*{} **\crv{(-16,-6) & (-8,6)}?(1)*\dir{}+(.1,0)*{\bullet};
  (-8,-10)*{}; (0,10)*{} **\crv{(-8,-6) & (-0,6)}?(.6)*\dir{}+(.2,0)*{\bullet}?
  (1)*\dir{}+(.1,0)*{\bullet};
  (0,10)*{}; (-16,30)*{} **\crv{(0,14) & (-16,26)}?(1)*\dir{>};
 (-16,10)*{}; (-8,30)*{} **\crv{(-16,14) & (-8,26)}?(1)*\dir{>};
  (-8,10)*{}; (0,30)*{} **\crv{(-8,14) & (-0,26)}?(1)*\dir{>}?(.6)*\dir{}+(.25,0)*{\bullet};
   }};
 (0,-5)*{\lambda};
 \endxy}
\]
\end{itemize}
\end{defn}

$\Ucat$ has graded 2-homs defined by
\begin{equation}
  \HOMU(x,y) := \bigoplus_{t \in \Z} \Hom_{\Ucat}(x\{t\},y).
\end{equation}


\section{The 2-category of universal $\mathfrak{sl}_3$-foams}
\label{sec:foams}

In this section we define a 2-category of foams related to the universal 
$\mathfrak{sl}_3$-foams defined in \cite{MV}. 
\subsection{The category $\mathbf{Foam}_{/\ell}$}
First let us recall the definition and some basic facts about the 
category of universal $\mathfrak{sl}_3$ foams $\mathbf{Foam}_{/\ell}$ from 
\cite{MV}. A \emph{closed web} is an oriented trivalent graph such that 
the edges meeting at a trivalent vertex are all oriented away from the vertex 
or are all oriented towards it. A \emph{foam} is a cobordism with singular 
arcs between two closed webs. A singular arc in a foam $f$ is the set of points of $f$ that have a neighborhood homeomorphic to the letter Y times an interval (see the examples in Figure~\ref{fig:ssaddle}).

\begin{figure}[h]
\center
{
\includegraphics[height=0.5in]{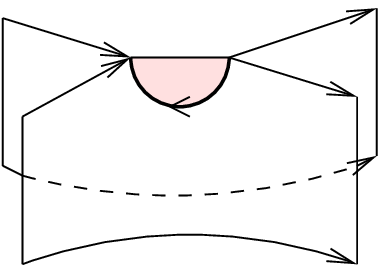}\qquad
\includegraphics[height=0.5in]{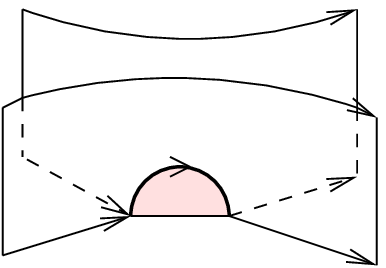}
}
\caption{The zip and the unzip}
\label{fig:ssaddle}
\end{figure}
\noindent Interpreted as morphisms, we read foams from bottom to top by convention, and the orientation of the singular arcs is by convention as in Figure~\ref{fig:ssaddle}. Foams can have dots that can move freely on the facet to which they belong, but are not allowed to cross singular arcs. Let $\F_2[a,b,c]$ be the ring of 
polynomials in $a,b,c$ with coefficients in $\F_2$. 

\begin{defn}
${\bf Foam}$ is the category whose objects are closed webs and whose morphisms 
are $\F_2[a,b,c]$-linear 
combinations of isotopy classes of foams.
\end{defn}

\noindent ${\bf Foam}$ is an additive category.

In order to construct the universal theory we divide $\mathbf{Foam}$ by the local relations  $\ell=(3D, CN, S, \Theta)$ below. 
$$\xymatrix@R=2mm
{
\raisebox{-5pt}{
\includegraphics[height=0.2in]{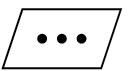}}=
a\raisebox{-5pt}{
\includegraphics[height=0.2in]{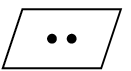}}+
b\raisebox{-5pt}{
\includegraphics[height=0.2in]{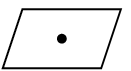}}+
c\raisebox{-5pt}{
\includegraphics[height=0.2in]{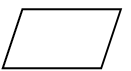}}
 & \text{(3D)}
\\
\raisebox{-17pt}{
\includegraphics[height=0.5in,width=0.2in]{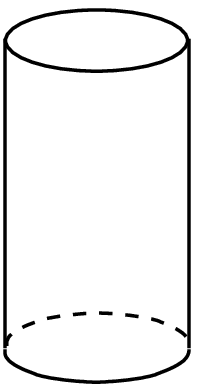}}=
\raisebox{-17pt}{
\includegraphics[height=0.5in,width=0.2in]{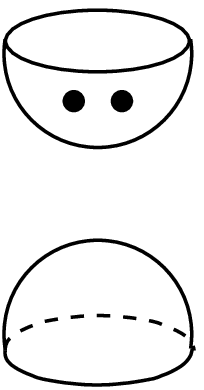}}+
\raisebox{-17pt}{
\includegraphics[height=0.5in,width=0.2in]{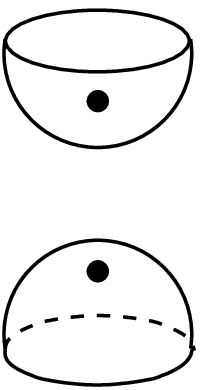}}+
\raisebox{-17pt}{
\includegraphics[height=0.5in,width=0.2in]{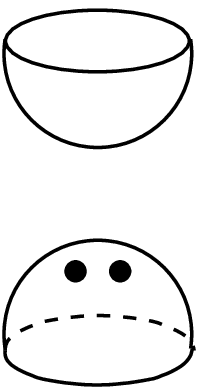}}
+a
\left( 
\raisebox{-17pt}{
\includegraphics[height=0.5in,width=0.2in]{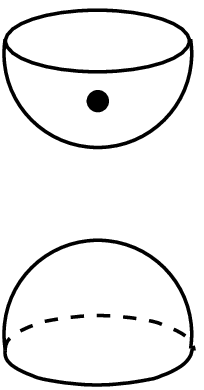}}+
 \raisebox{-17pt}{
\includegraphics[height=0.5in,width=0.2in]{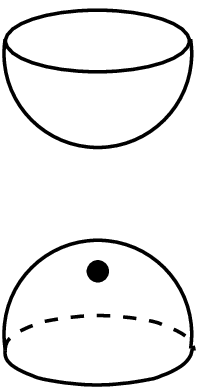}}
\right)
+b
\raisebox{-17pt}{
\includegraphics[height=0.5in,width=0.2in]{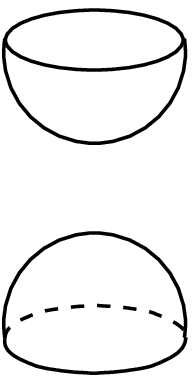}}
 & \text{(CN)} \\
\raisebox{-8pt}{\includegraphics[width=0.3in]{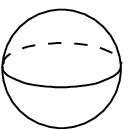}}=
\raisebox{-8pt}{\includegraphics[width=0.3in]{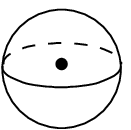}}=0,\quad
\raisebox{-8pt}{\includegraphics[width=0.3in]{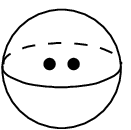}}=1
 & \text{(S)}
}$$

\begin{figure}[h]
\center
{
\includegraphics[height=0.4in]{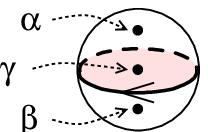}
}
\end{figure} 
Theta-foams are obtained by gluing three oriented disks along their boundaries (their orientations must coincide), as shown on the right. Note the orientation of the singular circle. Let $\alpha$, $\beta$, $\gamma$ denote the number of dots on each facet.
The ($\Theta$) \emph{relation} says that for $\alpha$, $\beta$ or $\gamma\leq 2$
\begin{equation*}
\theta(\alpha,\beta,\gamma)=\left\{
\begin{array}{cl}
1 & (\alpha,\beta,\gamma)=(1,2,0)\mbox{ or any permutation} \\ 
0 & \mbox{else}
\end{array}
\right.
\qquad(\Theta)
\end{equation*}
Reversing the orientation of the singular circle gives the same values for 
$\theta(\alpha,\beta,\gamma)$. 
Note that when we have three or more dots on a facet of a foam we can use the (\emph{3D}) relation to reduce to the case where it has less than three dots.

A closed foam $f$ can be viewed as a morphism from the empty web to itself which by the relations  (\emph{3D}, \emph{CN}, \emph{S}, $\Theta$) is an element of $\F_2[a,b,c]$. It can be checked that this set of relations is consistent and determines uniquely the evaluation of every closed foam 
$f$, denoted $C(f)$. Define a $q$-grading on $\F_2[a,b,c]$ as $q(1)=0$, $q(a)=2$, $q(b)=4$ and $q(c)=6$. We define the \emph{q-degree} of a foam $f$ with 
$d$ dots by $$q(f)=-2\chi(f)+ \chi(\partial f)+2d,$$ where $\chi$ denotes the 
Euler characteristic and $\partial f$ the boundary of $f$.

\begin{defn}
${\bf Foam}_{/\ell}$ is the quotient of the category ${\bf Foam}$ by the local relations $\ell$. 
For webs $\Gamma$, $\Gamma'$ and for families $f_i\in\Hom_{\mathbf{Foam}_{/\ell}}(\Gamma,\Gamma')$ and 
$c_i\in\F_2[a,b,c]$ we impose $\sum_ic_if_i=0$ if and only if 
$\sum_ic_iC(g'f_ig)=0$ holds, for all 
$g\in\Hom_{\mathbf{Foam}_{/\ell}}(\emptyset,\Gamma)$ and 
$g'\in\Hom_{\mathbf{Foam}_{/\ell}}(\Gamma',\emptyset)$.
\end{defn}

\begin{lem}\label{lem:identities}
We have the following relations in ${\bf Foam}_{/\ell}$:
$$\xymatrix@R=2mm{
\raisebox{-16pt}{\includegraphics[width=0.6in]{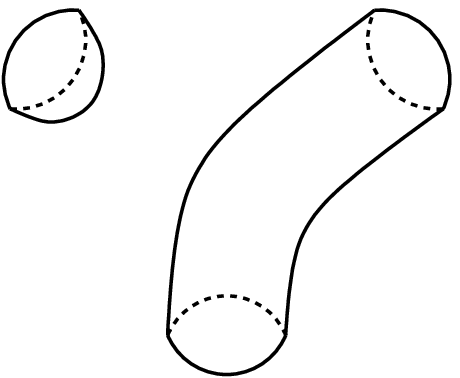}}+
\raisebox{-16pt}{\includegraphics[width=0.6in]{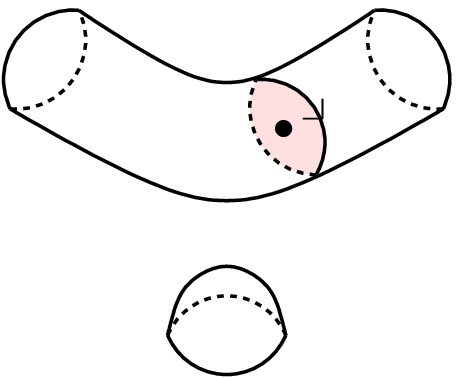}}=
\raisebox{-16pt}{\includegraphics[width=0.6in]{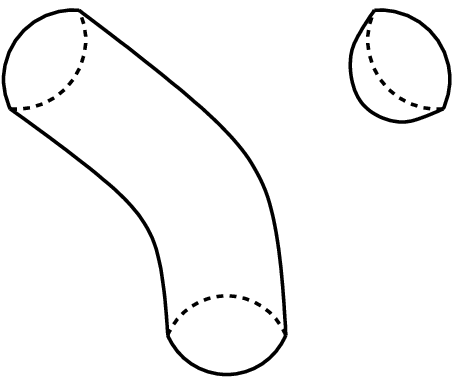}}+
\raisebox{-16pt}{\includegraphics[width=0.6in]{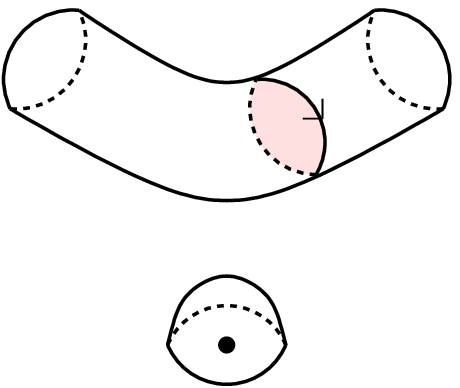}}
 & \text{(4C)}  \\
\label{R-rel}
\raisebox{-16pt}{\includegraphics[height=0.5in,width=0.2in]{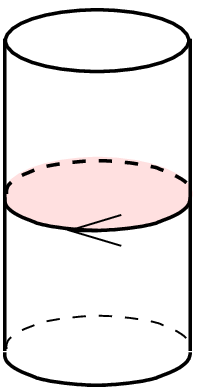}}=
\raisebox{-16pt}{\includegraphics[height=0.5in,width=0.2in]{cnecka1}}+
\raisebox{-16pt}{\includegraphics[height=0.5in,width=0.2in]{cnecka2}}
 & \text{(RD)}  \\
\raisebox{-20pt}{\includegraphics[height=0.6in]{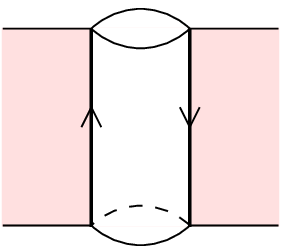}}=
\raisebox{-26pt}{\includegraphics[height=0.75in]{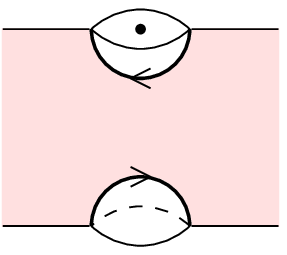}}+
\raisebox{-26pt}{\includegraphics[height=0.75in]{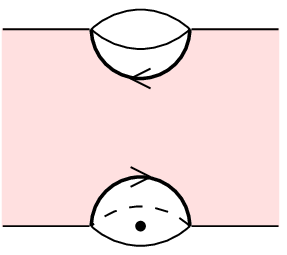}}
 & \text{(DR)}  \\
\raisebox{-28pt}{\includegraphics[height=0.8in]{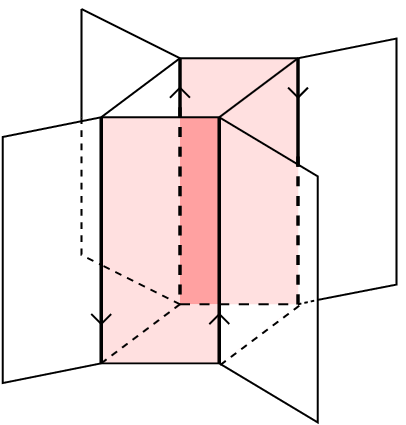}}=  
\raisebox{-28pt}{\includegraphics[height=0.8in]{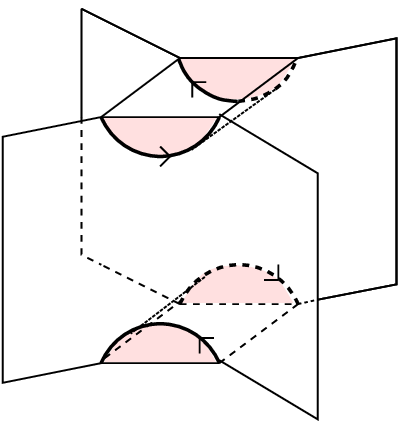}}+ 
\raisebox{-28pt}{\includegraphics[height=0.8in]{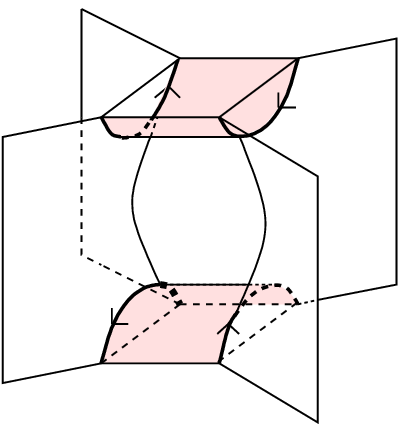}}
 & \text{(SqR)}
}$$
\end{lem}

\noindent In Figure~\ref{fig:pdots} we also have a set of useful identities, called \emph{dot migration}, 
which establish the way we can exchange dots between faces. These identities can be used for the simplification of foams and are an immediate consequence of the relations in $\ell$.
\begin{figure}[h]
$$\xymatrix@R=1mm{
\raisebox{-0.27in}{\includegraphics[height=0.5in]{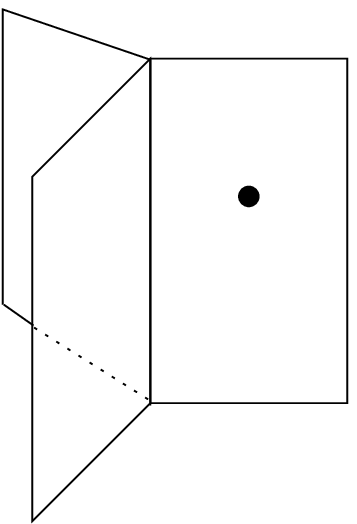}}
\,+\,
\raisebox{-0.27in}{\includegraphics[height=0.5in]{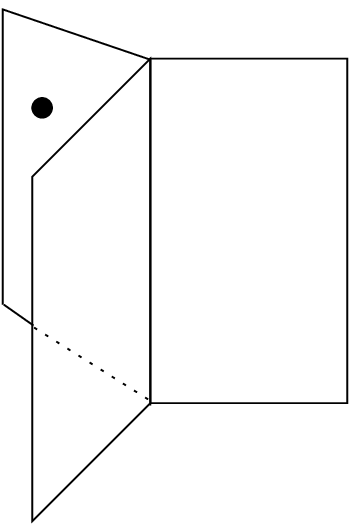}}
\,+\,
\raisebox{-0.27in}{\includegraphics[height=0.5in]{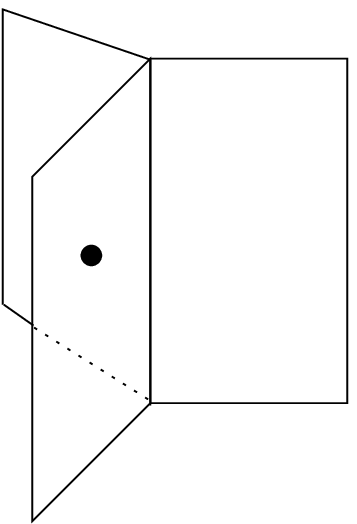}}
\,=\,a\,
\raisebox{-0.27in}{\includegraphics[height=0.5in]{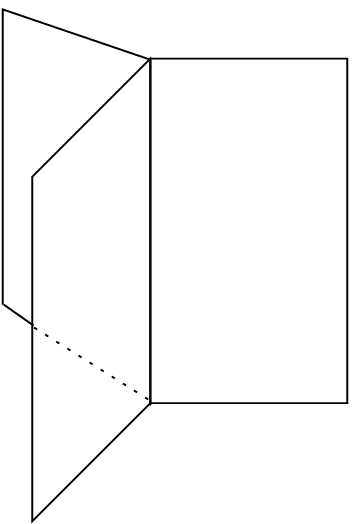}} 
\\
\raisebox{-0.27in}{\includegraphics[height=0.5in]{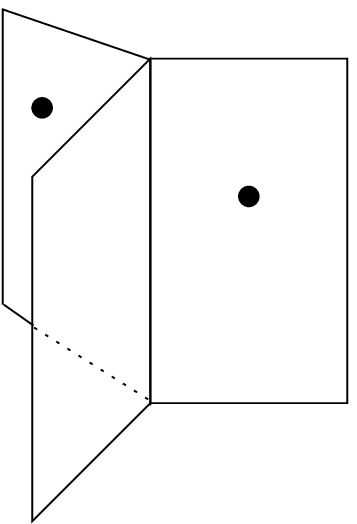}}
\,+\,
\raisebox{-0.27in}{\includegraphics[height=0.5in]{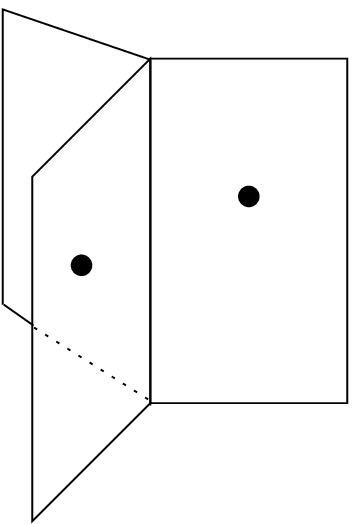}}
\,+\,
\raisebox{-0.27in}{\includegraphics[height=0.5in]{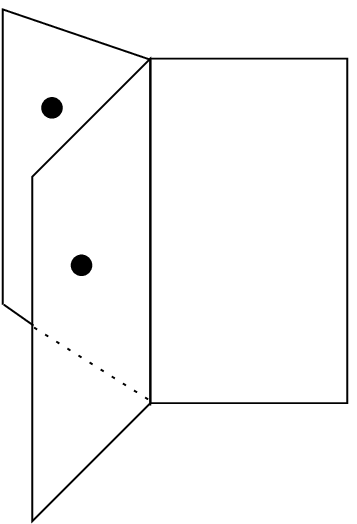}}
\,=\,b\,
\raisebox{-0.27in}{\includegraphics[height=0.5in]{pdots000}} 
\\
\raisebox{-0.27in}{\includegraphics[height=0.5in]{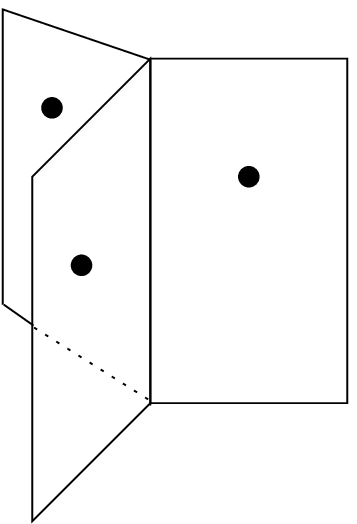}}
\,=\,c\,
\raisebox{-0.27in}{\includegraphics[height=0.5in]{pdots000}}
}$$
\caption{Dot migration}
\label{fig:pdots}
\end{figure}

\noindent One can generalize ${\bf Foam}_{/\ell}$ to open webs and foams with corners. The 
local relations are the same and a linear combination of foams with corners 
is zero if and only if for any way of closing the foams we get a linear 
combination of closed foams which is equal to zero. The grading formula has 
to be slightly changed, as we will show in the next section, where we define 
a 2-category of foams with corners more precisely.  

\subsection{The 2-category $\mathbf{2Foam}_{/\ell}$}

In order to define foamation as a 2-functor we have to define a 2-category of 
$\mathfrak{sl}_3$-foams closely related to $\mathbf{Foam}_{/\ell}$. 

\begin{defn}
Let $\FF$ be the 2-category whose objects are finite sequences of points. 
We take the points to lie along the 
y-axis with a fixed distance between them, e.g. each point has coordinates 
$(0,n,0)$ for some $n\in \N$.   

Let the 1-morphisms be finite sequences of vertical line segments possibly 
linked by horizontal line segments. Each vertical line segment is allowed to 
be a finite disjoint union of smaller subsegments of the 
form $\{0,n,z)\colon a\leq z\leq b\}$ for some fixed $n\in\N$ and 
$0\leq a<b\leq 1$.
Except for the bottom and top endpoints, which have $z$-coordinate equal to 
$0$ and $1$ respectively, each endpoint of a vertical line 
subsegment is also the endpoint of a horizontal line segment. Furthermore 
there might be points in the interior of a subsegment which are endpoints 
of horizontal line segments. We consider such a line segment to be a 
1-morphism from its 
intersection with the $x=z=0$ axis to its intersection with the line 
$\{(0,y,1)\colon y\in\R\}$ 
(after adjusting the z-variable of the latter points we get an object as 
defined above). Composition is defined by vertical glueing and rescaling. 
If one vertical line ends with a hole and the other starts with a hole, the 
composite will be a line with a hole at the place of the two old holes, i.e. 
the glueing does not close the holes. If one vertical line ends with hole and 
another starts without one, the two are non-composable.   

A 2-morphism is a foam whose boundary consists of a finite sequence of squares 
whose vertices are of the form 
$\{(x,n,z)\colon x\in\{-1,0\},\,z\in\{0,1\}\}$, for some $n\in\N$, 
and oriented horizontal line segments between their vertical 
parts. The vertical parts of the squares are allowed to have holes in them, 
the horizontal ones are not. 
The foams are required to be such that when we take the intersection with the 
$x=0$ and $x=-1$ planes, the source and target of the 2-morphism, we 
get 1-morphisms as described above. As in 
the previous subsection the foams are subject to the relations in $\ell$ 
and two open foams are equivalent if and only if all their possible closures 
are the same modulo $\ell$. 
See Figure~\ref{Fig:example} for an example. 
\begin{figure}[h!]
\hair 2pt
\centering
\figins{-30}{0.75}{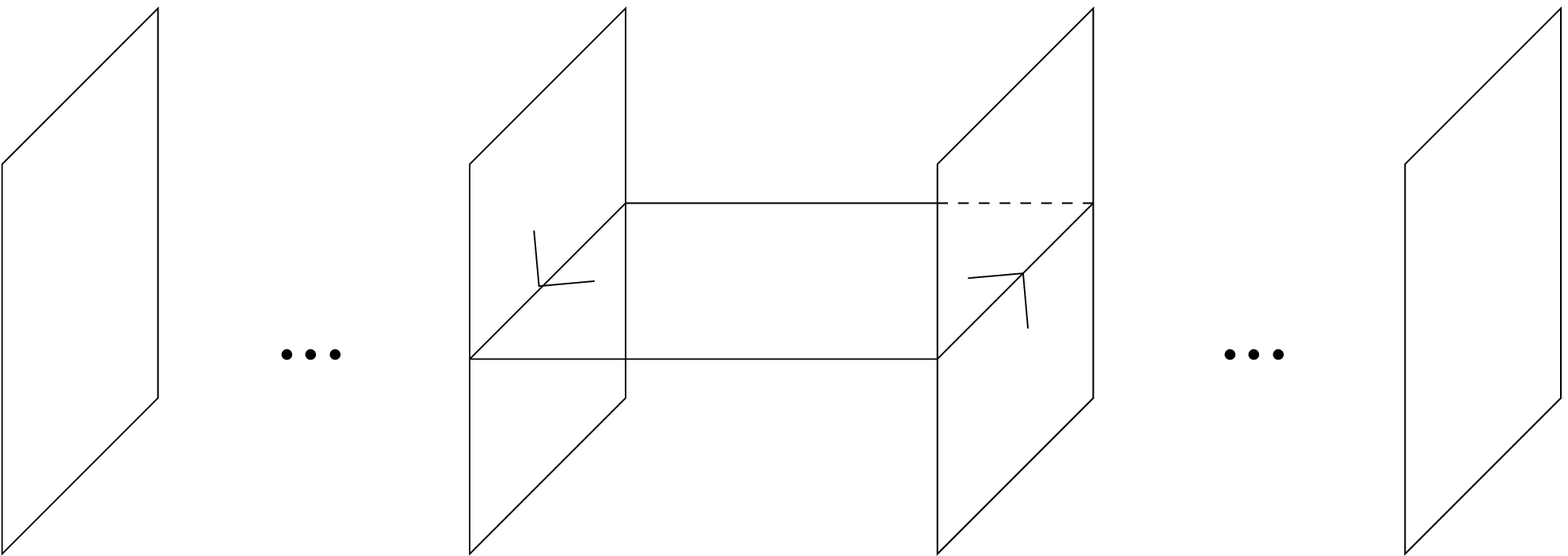}
\caption{Example of a 2-morphism in $\FF$}
\label{Fig:example}
\end{figure}
\noindent Horizontal composition is 
defined by horizontal glueing and rescaling and vertical composition by 
vertical glueing and rescaling. As 
above ``holes can only be glued to holes''.   
\end{defn}

\noindent Note that $\FF$ is a strict 2-category. We also define a grading on 
$\FF$, i.e. the 2-morphisms have degrees. This 
grading corresponds to the one given in the previous section 
with one extra term because the foams have corners. 

\begin{defn} The degree of a 2-morphism $f$ in $\FF$ is given by the 
formula
$$q(f)=-2\chi(f)+\chi(\partial f)+2d+\beta,$$ 
where $\beta$ is the number of vertical boundary components of 
$f$.
\end{defn} 

\noindent Note that $q$ is additive under vertical and horizontal composition. 
Note also that relation (S) implies the following result, which we will need 
in the following section.  
\begin{lem}
\label{lem:negdots}
Dotted spheres of negative degree are zero. 
\end{lem}

  
\section{Foamation}
\label{sec:hol}

In this section we define 2-functors $\HH_s\colon\Ucat\to \FF$, for 
$s\in \{0,1,2,3\}$. The main idea is that the 
image of a 2D-picture in $\Ucat$ should be a 3D-foam in $\FF$ such that 
vertical cross sections of that foam 
correspond to the initial 2D-picture. How to do that precisely is the content 
of the following definition. In the sequel we keep $k$ fixed in the definition 
of $\Ucat$.   

\begin{defn} 
On objects we define $\HH_s(\lambda)$ to be a subset of 
$\{(0,n,0)\colon n=0,\ldots,k-1\}$. To see which points appear in 
$\HH_s(\lambda)$ label the last point, i.e. $(0,k-1,0)$, by $s$. Label 
$(0,k-2,0)$ by $s-\lambda_1$, $(0,k-3,0)$ by $s-\lambda_1-\lambda_2$ etc. 
If one of the points has a label outside the range $\{0,1,2,3\}$, we define 
the image of $\lambda$ to be zero. Otherwise, remove the points labelled 
$0$ or $3$. The image of $\lambda$ is given by the remaining points.    

On 1-morphisms we define $\HH_s(\cal{E}_{\ii}1_{\lambda})$ as follows: 
first take $k$ vertical line segments $v_n$, where $n=0,\ldots,k-1$, with lower 
vertices $\{(0,k-1-n,0)\colon n=0,\ldots,k-1\}$ and upper vertices 
$\{(0,k-1-n,1)\colon n=0,\ldots,k-1\}$. Suppose 
$\cal{E}_{\ii}=\cal{E}_{\epsilon_{1}i_1}\cdots\cal{E}_{\epsilon_{m}i_m}$. 
Let $h_j=1-j/(m+1)$. For each $\cal{E}_{\epsilon_j i_j}$ add a horizontal line 
segment with height $h_j$ between $v_{i_j-1}$ and $v_{i_j}$. If $\epsilon_j=+$ 
orient the horizontal line segment to the right, if not orient it to the left. 
Then proceed as follows: 
label the lower vertex of $v_0$ by $s_{0}=s$ and the lower vertex of 
$v_n$ by $s_n=s-\lambda_1-\cdots-\lambda_{n}$ for $n=1,\dots,k-1$. On each 
$v_n$ this label remains constant until we meet a point where $v_n$ is glued 
to a 
horizontal line segment (when moving upwards). When we pass such a glueing 
point the label changes by $+1$ if the horizontal line segment is oriented 
towards $v_n$ and by $-1$ if not. This way all line subsegments of the 
$v_n$ are labelled. If there is a subsegment which has a label outside the 
range $\{0,1,2,3\}$ we define $\HH_s(\cal{E}_{\ii}1_{\lambda})=0$. If not, 
remove the subsegments with label $0$ or $3$ and take 
$\HH_s(\cal{E}_{\ii}1_{\lambda})$ to be the remaining graph. 

Finally, on 2-morphisms we define $\HH_{s}$ as follows:


\begin{figure}[h!]
\hair 2pt
\labellist
\pinlabel $s_{k-1}$ at 0 -50
\pinlabel $s_{i}$ at 180 -50
\pinlabel $s_{i-1}$ at 360 -50
\pinlabel $s$ at 540 -50
\pinlabel $s$ at 540  230
\pinlabel $s_{i}-1$ at 180  230 
\pinlabel $s_{i-1}+1$ at 360  230
\pinlabel $s_{k-1}$ at 0  230
\endlabellist
\centering
\raisebox{0pt}{\xy
 (0,8);(0,-8); **\dir{-} ?(.5)*\dir{>}+(2.3,0)*{\scriptstyle{}};
 (0,-11)*{ i};(0,11)*{ i};
 (6,2)*{ \lambda};
 (-8,2)*{ \lambda +i_X};
 (-10,0)*{};(10,0)*{};
 \endxy}$\quad\mapsto\quad$ \figins{-30}{0.75}{foamidentotal}
\end{figure}

\vskip30pt

\begin{figure}[h!]
\hair 2pt
\labellist
\pinlabel $s_{k-1}$ at 0 -50
\pinlabel $s_{i}$ at 180 -50
\pinlabel $s_{i-1}$ at 360 -50
\pinlabel $s$ at 540 -50
\pinlabel $s$ at 540  230
\pinlabel $s_{i}-1$ at 180  230 
\pinlabel $s_{i-1}+1$ at 360  230
\pinlabel $s_{k-1}$ at 0  230
\endlabellist
\centering
\raisebox{0pt}{\xy $\Uup_{i,\lambda}$ \endxy}$\quad\quad\quad\mapsto\quad$ \figins{-30}{0.75}{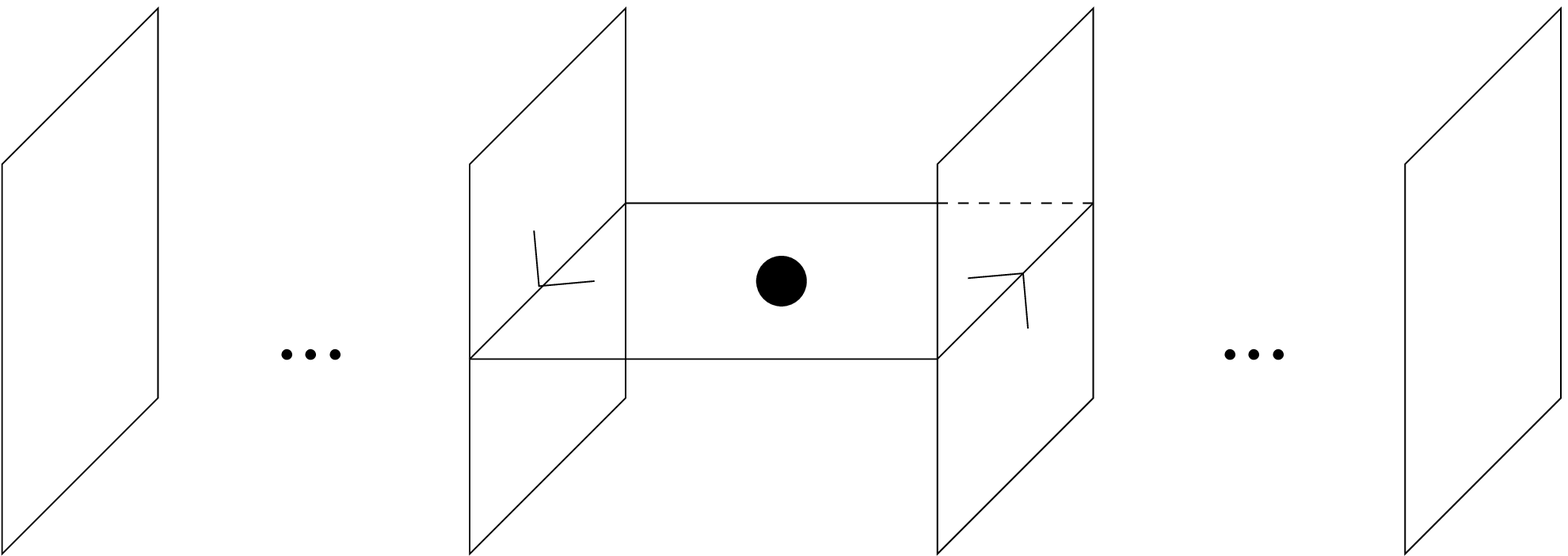}
\end{figure}
\vskip20pt
\noindent We have indicated the labels of the facets, with the $s_n$ as 
defined above. 

Suppose $\cal{E}_{\ii}=\cal{E}_{\epsilon_{1}i_1}\cdots\cal{E}_{\epsilon_{m}i_m}$. 
Let $h_j=1-(j/m+1)$. We define $H_s(1_{\cal{E}_{\ii}1_{\lambda}})$ as above, but with more 
horizontal sheets. To be precise, for each $j$, we add a horizontal sheet with 
height $h_j$ between the $i_j-1$-th and the $v_{i_j}$th vertical sheets. 
If $\epsilon_j=+$ we orient the horizontal sheet as shown in the picture 
above, if not we give it the opposite orientation. The labellings of the 
facets are as indicated above, and depend on the orientation of the horizontal 
sheets. If one of the labels is outside the range 
of $\{0,1,2,3\}$, then the image of $1_{\cal{E}_{\ii}1_{\lambda}}$ is zero. If all 
labels are within range, remove 
the facets labelled $0$ or $3$. The image of $1_{\cal{E}_{\ii}1_{\lambda}}$ is the 
remaining foam. Let us do one concrete example to clarify our definition. 
Suppose $k=3$, $s=2$ and 
$\cal{E}_{(+1,-2)}=\cal{E}_{+1}\cal{E}_{-2}$ and $\lambda=(0,0)$. 
Then $H_2(1_{\cal{E}_{(+1,-2)}1_{(0,0)}})$ is given by the foam obtained from 
\begin{figure}[h!]
\hair 2pt
\labellist
\pinlabel $2$ at 10 30
\pinlabel $2$ at 180 30
\pinlabel $3$ at 10  137 
\pinlabel $2$ at 350 30
\pinlabel $1$ at 180 80
\pinlabel $0$ at 180 137
\pinlabel $3$ at 350 137
\endlabellist
\centering
\raisebox{0pt}{\figins{-30}{1.25}{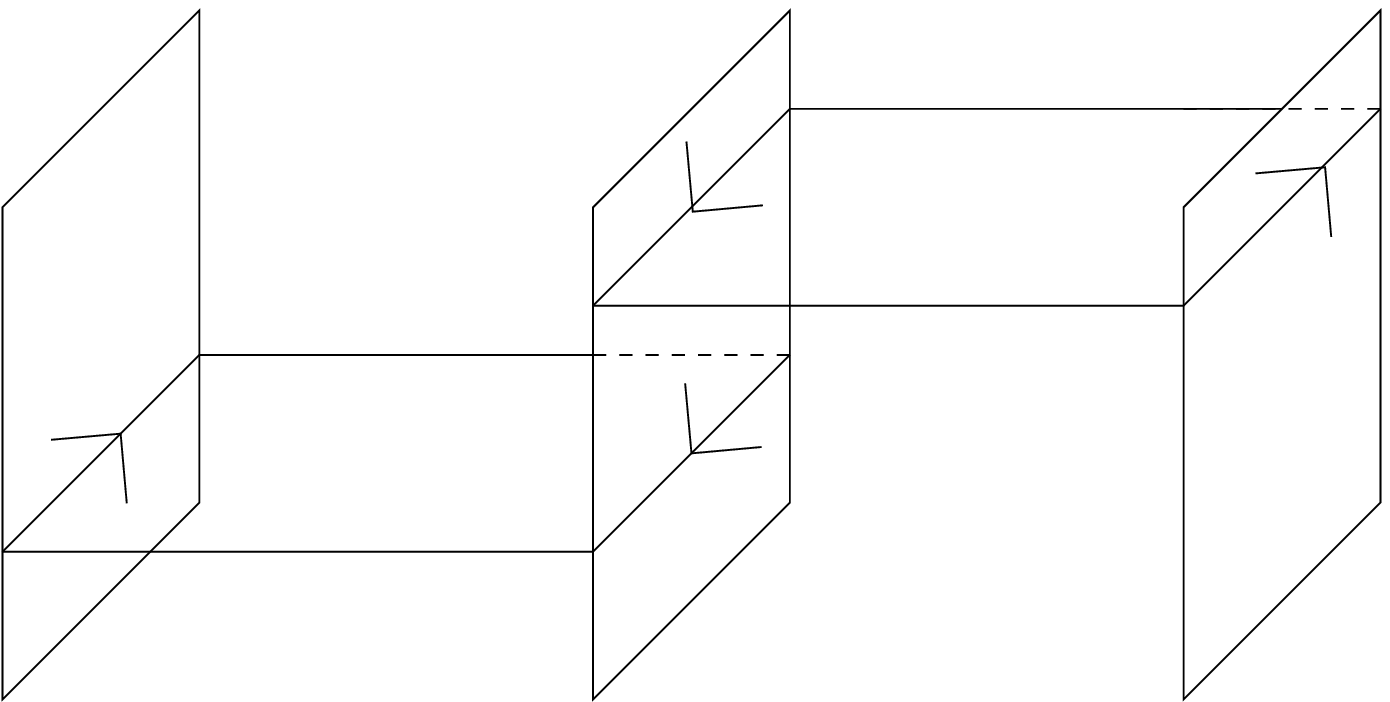}}
\end{figure}
\vskip1pt
\noindent by removing the facets labelled $0$ and $3$. Note that 
$$\cal{E}_{(+1,-2)}1_{(0,0)}=1_{(3,-3)}\cal{E}_{+1}1_{(1,-2)}\cal{E}_{-2}1_{(0,0)}$$
and that the pairs $(3,-3)$, $(1,-2)$ and $(0,0)$ correspond precisely to the 
differences between the labels on the right and left vertical sheets on the 
level above, between and under the two horizontal sheets respectively.  

Next we define $H_s$ for the other generating 2-morphisms. For shortness we 
omit the vertical sheets which are unimportant in the pictures. The labels 
are calculated as indicated in the pictures, and labels outside the indicated 
range give rise to a zero image. If all labels are within range, the facets 
labelled $0$ or $3$ have to be removed to obtain 
the final image under $H_s$. Dots on the 2-morphisms are mapped to dots 
on the corresponding facets. We only give the images of the 2-morphisms with 
one particular orientation. For the other orientations just reorient the foams 
and relabel the facets accordingly.   

\begin{figure}[h!]
\hair 2pt
\labellist
\pinlabel $s_{i}$ at 20 0
\pinlabel $s_{i-1}$ at 210 0
\pinlabel $s_{i}-2$ at 20 210 
\pinlabel $s_{i-1}+2$ at 180 210
\pinlabel $s_i-1$ at 90 75
\pinlabel $s_i-1$ at 125 130
\pinlabel $s_{i-1}+1$ at 280 80
\pinlabel $s_{i-1}+1$ at 310 125 
\endlabellist
\centering
\raisebox{0pt}{\xy $\Ucross_{i,i,\lambda}$ \endxy}$\qquad\qquad\quad\mapsto\quad$ 
\figins{-30}{1.25}{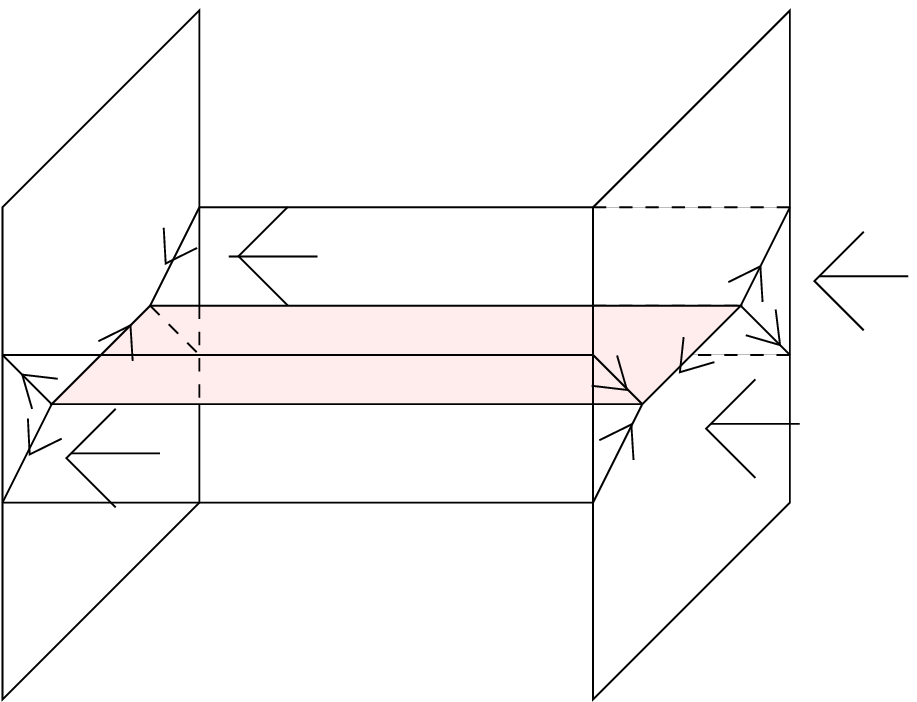}
\end{figure}
\vskip20pt

\begin{figure}[h!]
\hair 2pt
\labellist
\pinlabel $s_{i+1}$ at 30 0
\pinlabel $s_{i}$ at 190 0
\pinlabel $s_{i-1}$ at 375 0
\pinlabel $s_{i}$ at 210  210 
\pinlabel $s_{i-1}+1$ at 350  210
\pinlabel $s_{i+1}-1$ at 10  210
\pinlabel $s_i+1$ at 90 75
\pinlabel $s_i+1$ at 310 135
\endlabellist
\centering
\raisebox{0pt}{\xy $\Ucross_{i,i+1,\lambda}$ \endxy}$\qquad\qquad\quad\mapsto
\quad$ \figins{-30}{1.25}{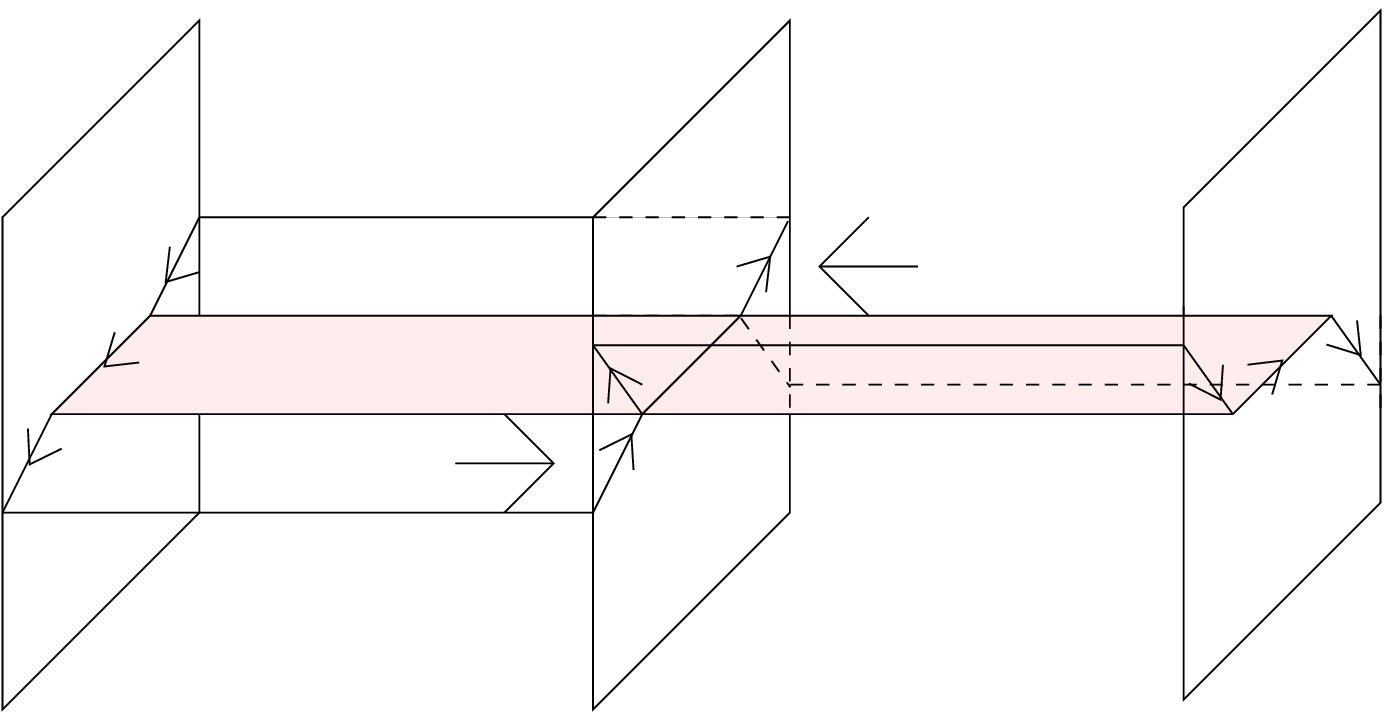}
\end{figure}
\vskip20pt
\noindent We define $\Ucross_{i-1,i,\lambda}$ similarly. 

If $j>i+1$ we define 

\begin{figure}[h!]
\hair 2pt
\labellist
\pinlabel $s_{j}$ at 20 0
\pinlabel $s_{j-1}$ at 200 0
\pinlabel $s_{i}$ at 360 0
\pinlabel $s_{j-1}+1$ at 175  210 
\pinlabel $s_{i}-1$ at 350  210
\pinlabel $s_{j}-1$ at 10  210
\pinlabel $s_{i-1}$ at 540 0
\pinlabel $s_{i-1}+1$ at 510 205
\endlabellist
\centering
\raisebox{0pt}{$\Ucross_{i,j,\lambda}$}$\qquad\mapsto
\quad$ \figins{-30}{1.25}{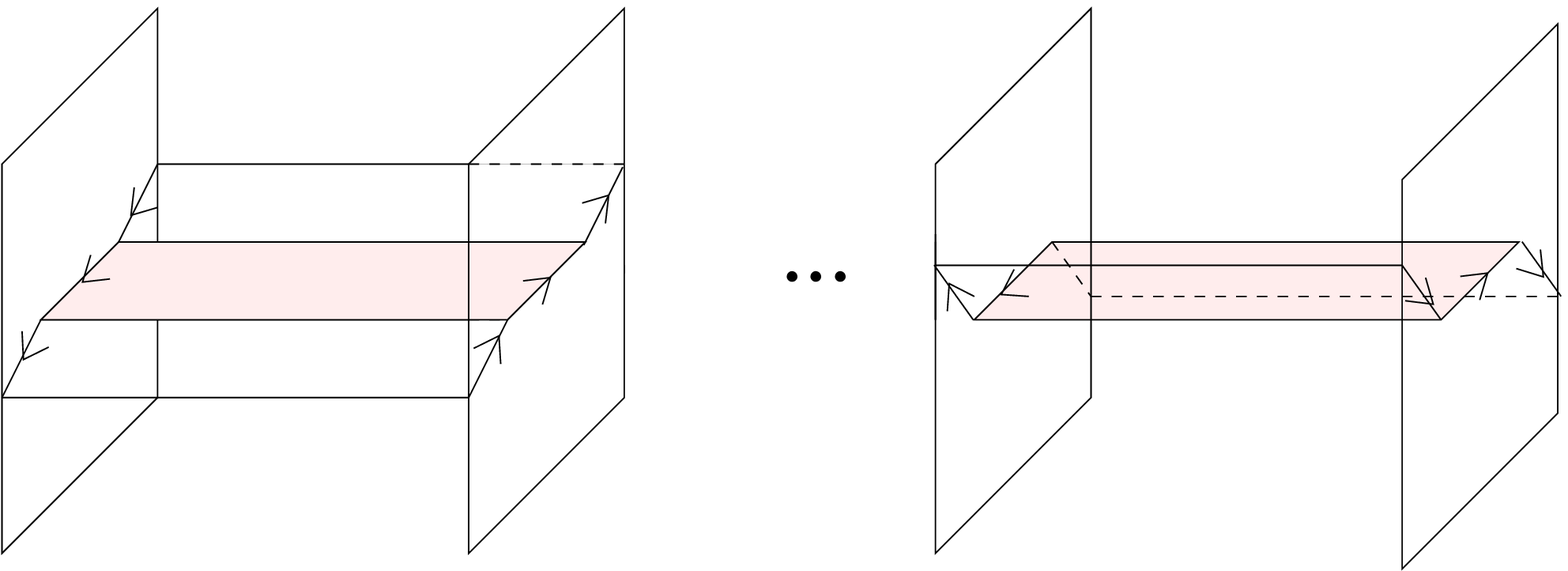}
\end{figure}
\vskip20pt

\noindent The images of the cups and caps are defined as follows:
\begin{figure}[h!]
\hair 2pt
\labellist
\pinlabel $s_{i}$ at 20 0
\pinlabel $s_{i-1}$ at 200 0
\pinlabel $s_{i-1}+1$ at 145 130
\pinlabel $s_i-1$ at 310 130
\endlabellist
\centering
\raisebox{0pt}{\xy $\Ucupl_{i,\lambda}$ \endxy}$\qquad\quad\mapsto
\quad$ \figins{-30}{1.25}{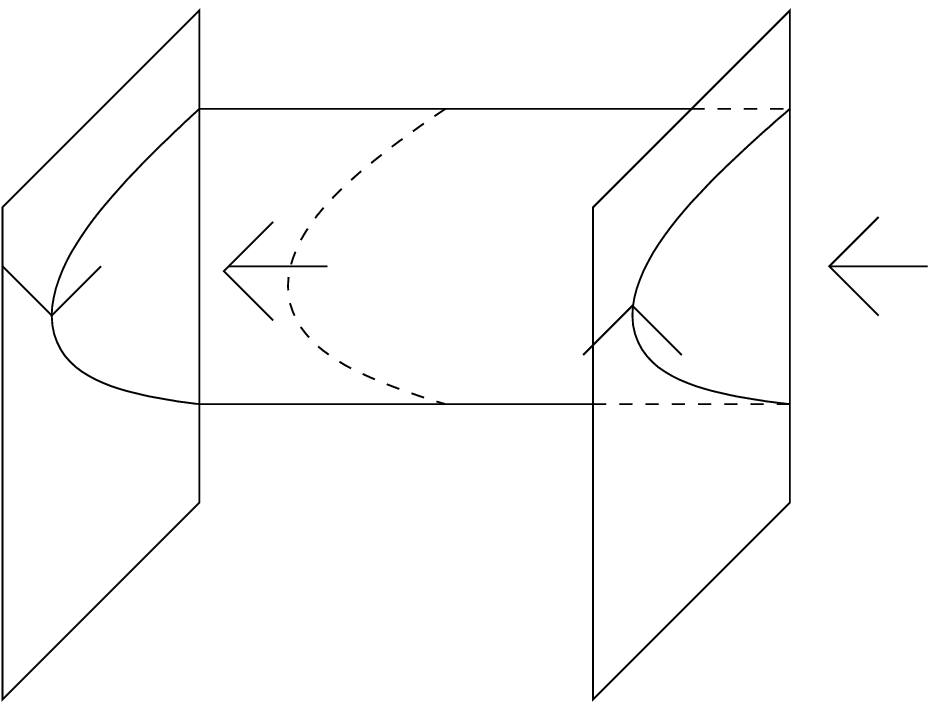}
\end{figure}
\vskip20pt

\begin{figure}[h!]
\hair 2pt
\labellist
\pinlabel $s_{i}$ at 20 0
\pinlabel $s_{i-1}$ at 210 0
\pinlabel $s_{i-1}+1$ at 295 70
\pinlabel $s_i-1$ at 110 70
\endlabellist
\centering
\raisebox{0pt}{\xy $\Ucapl_{i,\lambda}$ \endxy}$\quad\qquad\mapsto\quad$ 
\figins{-30}{1.25}{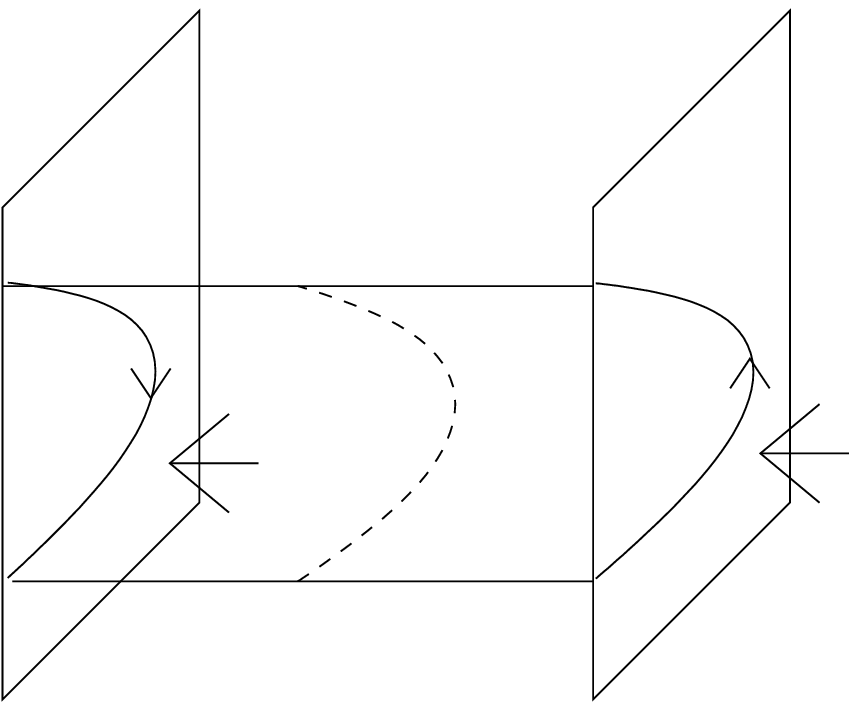}
\end{figure}
\vskip20pt
\end{defn}
We are now ready to prove the main theorem of our paper. 
\begin{thm}
\label{thm:main}
$\HH_s\colon \Ucat\to \FF$ is a degree preserving strict $\F_2$-linear 
2-functor, for any $s\in\{0,1,2,3\}$. 
\end{thm}
\begin{proof} Clearly the composition rules are preserved by $\HH_s$, 
for any $s\in\N$. It is also not hard to check that $\HH_s$ preserves the 
degrees. Note that one has to be careful and remove facets labelled $0$ or $3$ 
before computing the degree. 

What is left to prove is that $\HH_s$ respects the 
relations in $\Ucat$. We first prove this for the $\mathfrak{sl}_2$ relations. 
Relations \ref{eq_biadjoint1}, \ref{eq_biadjoint2} and \ref{eq_cyclic_dot} 
are preserved because $\HH_s$ maps the left and right sides of each relation 
to isotopic foams.

To see that relations~\ref{eq_positivity_bubbles} are preserved we have to 
realize that the image of a bubble with dots is given by $k-1$ vertical 
sheets with a tube with dots between two of them, possibly with some vertical 
facets removed because they are labelled $0$ or $3$. 
By the relations (CN) and (RD) in $\FF$ 
such a foam is equivalent to a $\F_2[a,b,c]$-linear combination of foams which are given 
by vertical sheets with dots and without holes, and spheres with dots. 
Since $\HH_s$ preserves degrees and $\F_2[a,b,c]$ has non-negative grading 
only, each of these foams in the linear combination has to have 
negative degree. Vertical sheets with dots have non-negative degree, so the 
spheres with dots have to have negative degree. Therefore, by 
Lemma~\ref{lem:negdots} everything has to be zero.    

The same line of reasoning shows that the image of a degree zero bubble equals 
the identity. In this case there is only one non-zero term in the linear 
combination, which is the term with dotless vertical sheets and degree zero 
sphere with dots. Therefore we get plus or minus the identity. 

For the next relations we have to determine the image of fake bubbles. Note 
that the fake bubbles are uniquely defined in terms of ordinary bubbles 
by the infinite Grassmannian relation~\ref{eq_infinite_Grass} and the 
additional condition which says that the degree zero fake bubbles are equal to 
one. Since we know the image of the ordinary bubbles, all we have to do is 
indicate foams which satisfy the infinite Grassmannian relation and 
satisfy the degree zero condition in $\FF$. It suffices to do this for 
$k=2$, since these relations are essentially $\mathfrak{sl}_2$ relations. 
For $\lambda\geq 0$ and $j< \lambda+1$, we put
\begin{figure}[h!]
\hair 2pt
\labellist
\pinlabel $\sum_{r=0}^{\lambda+1}$ at -40 100
\pinlabel $s-\lambda$ at 40 0
\pinlabel $s$ at 195 0
\pinlabel $r$ at 40 110
\pinlabel $j$ at 125 110
\pinlabel $\lambda+1-r$ at 253 110
\endlabellist
\centering
\raisebox{0pt}{\xy 0;/r.15pc/:
 (0,0)*{\ccbub{-\lambda-1+j}{}};
  (4,8)*{\lambda};
 \endxy}$\quad\mapsto\qquad\qquad$ 
\figins{-45}{1.40}{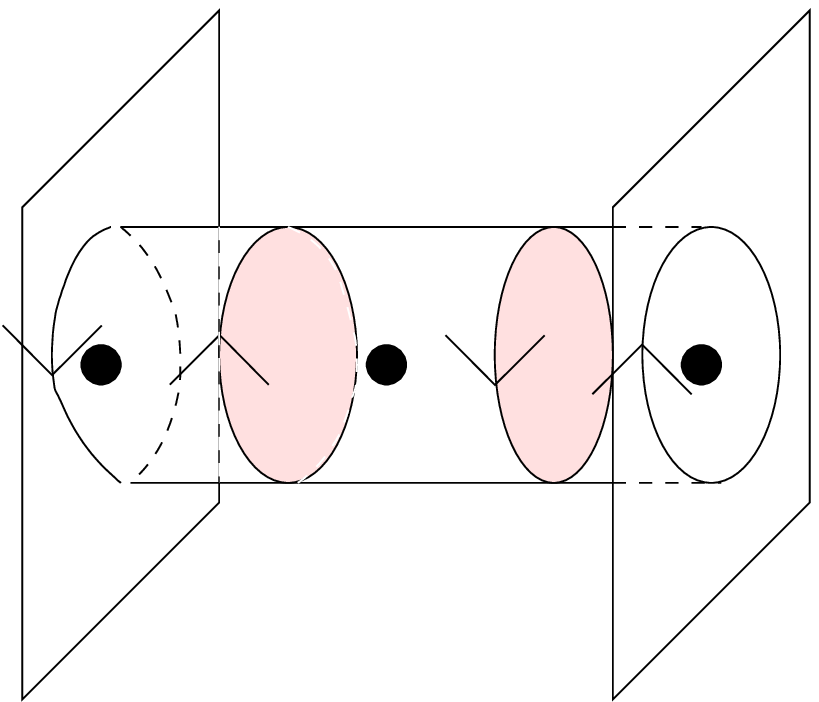}
\end{figure}
\vskip20pt
\noindent where a dot to the power $r$ means $r$ dots. Note that the labels 
of the interiors of the rightmost and leftmost discs are equal to $s-1$ 
and $s-\lambda+1$ respectively. If one of these labels is equal to $0$ or $3$ 
we have to remove the corresponding disc, so we cannot put dots on that disc. 
To avoid contradictions we simply assume that terms with a disc labelled $0$ 
or $3$ and a positive number of dots on that disc are equal to zero in the 
sum. By 
straightforward calculations one can check that with this definition 
the images of the 
fake bubbles satisfy the infininite Grassmannian relation and the degree 
zero condition. The images of the fake bubbles with the opposite orientation 
are defined likewise.       

Next we prove the right relation in~\ref{eq:redtobubbles}. The left one can 
be shown using similar arguments. Again, since this is an $\mathfrak{sl}_2$ 
relation we can assume that $k=2$ without loss of generality. First suppose 
that $\lambda<0$. Then the sum on the right-hand side of the equation is equal 
to zero, so we have to show that the image of the left-hand side is equal to 
zero as well. That image is given by 
\begin{figure}[h!]
\hair 2pt
\labellist
\pinlabel $s-\lambda+2$ at 50 0
\pinlabel $s$ at 190 0
\pinlabel $s+1$ at 190 200
\pinlabel $s-\lambda+1$ at 0 200
\endlabellist
\centering
\figins{-45}{1.40}{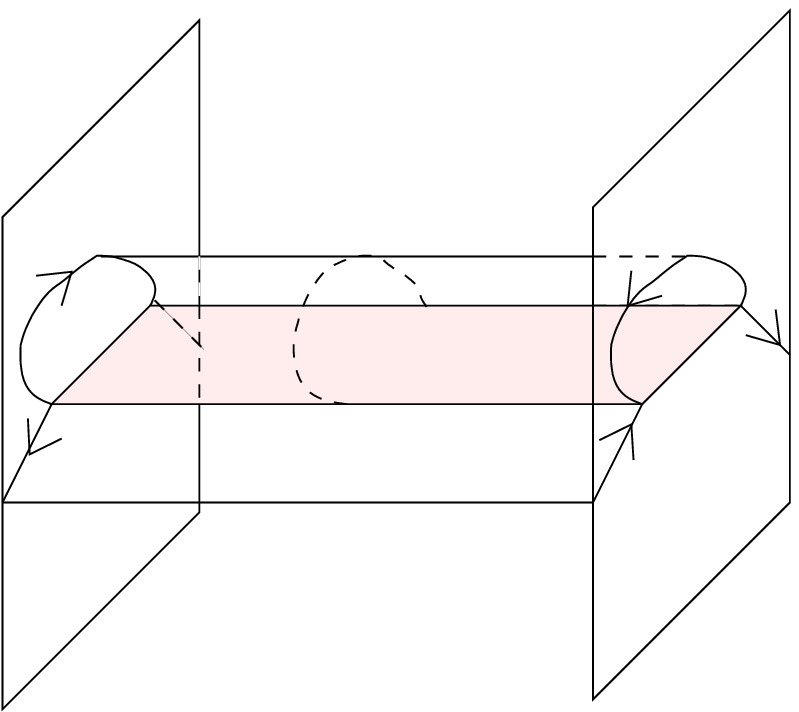}
\end{figure}
\vskip20pt

\noindent Note that the left and right inner discs are labelled $s-\lambda$ 
and $s+2$ respectively. Since we are assuming that $\lambda<0$ we see 
that this foam is zero for $s\ne 0$ because $s$ or $s-\lambda+2$ will be 
outside the range $\{0,1,2,3\}$. For the same reason the foam is zero when 
$s=0$ and $\lambda<-1$. If $s=0$ and $\lambda=-1$ the foam we get is a facet 
with a bubble without dots (after removing the facets labelled $0$ and $3$). 
By (CN) and ($\Theta$) such a foam is equal to zero. 

If $\lambda\geq 0$, we can apply (DR) to the singular tube in the 
picture. 
After removing possible bubbles with dots we get a linear combination 
of foams like 
\begin{figure}[h!]
\hair 2pt
\labellist
\pinlabel $s-\lambda+2$ at 50 0
\pinlabel $s$ at 190 0
\pinlabel $s+1$ at 190 200
\pinlabel $s-\lambda+1$ at 0 200
\endlabellist
\centering
\figins{-45}{1.40}{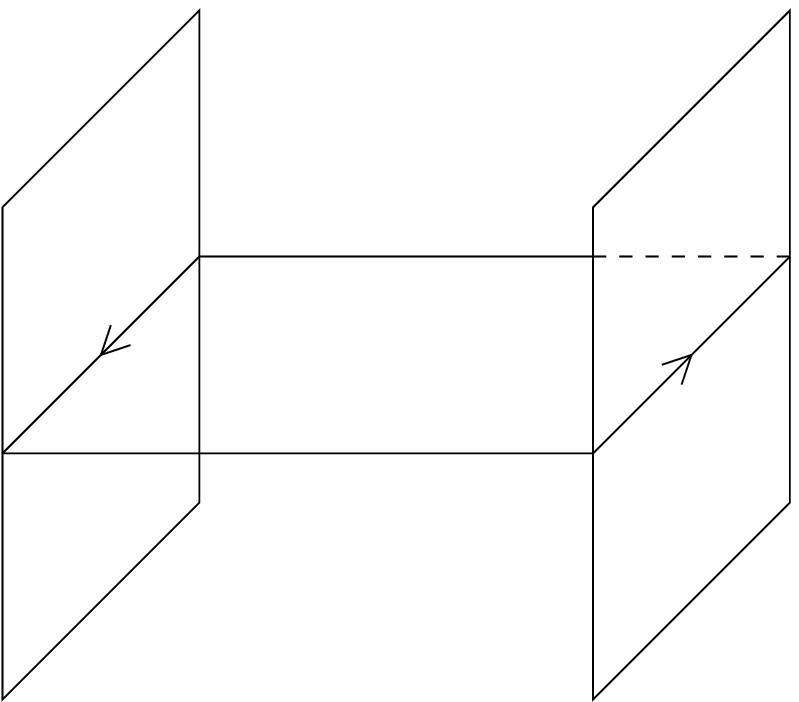}
\end{figure}
\vskip20pt 
\noindent with some dots.
\newpage 
The foam corresponding to the right-hand side of relation~\ref{eq:redtobubbles}
is a linear combination of foams of the form 
\begin{figure}[h!]
\hair 2pt
\labellist
\pinlabel $s-\lambda+2$ at 60 0
\pinlabel $s$ at 190 0
\pinlabel $s+1$ at 195 230
\pinlabel $s-\lambda+1$ at 5 230
\pinlabel $\alpha$ at 37 143 
\pinlabel $\beta$ at 128 143
\pinlabel $\gamma$ at 209 143
\pinlabel $\delta$ at 128 70
\endlabellist
\centering
\figins{-45}{1.50}{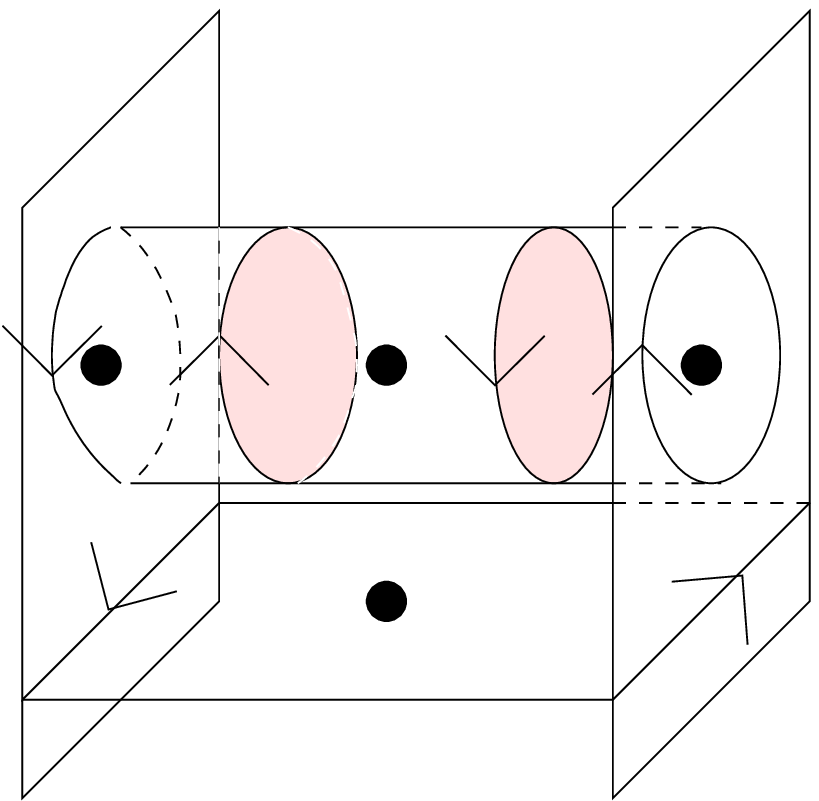}
\end{figure}
\vskip5pt
\noindent for certain non-negative integers $\alpha,\beta,\gamma$ and $\delta$. 
After applying (RD) to both internal discs and removing the bubbles and the 
spheres we again obtain a linear combination of foams as pictured in  
the second last figure with some dots. 
All these calculations are straightforward 
and one can check that in all cases the right relation 
in~\ref{eq:redtobubbles} is preserved by $\HH_s$ after moving the dots around 
using dot migration. 

Let us now show that relation~\ref{eq_ident_decomp} is preserved. We show this 
for the first relation, the second being similar. We can assume 
that $\lambda$ is a non-negative integer (instead of a sequence of integers). 
It suffices to consider the cases $s=2,3$ for $\lambda=0,1,2,3$. For $s=1,2$ 
the same arguments work because the foams in that case can be obtained from 
those for $s=2,3$ by applying a symmetry in the $xz$-plane which slices the 
foams in the middle. Therefore, let $s=2,3$. The left-hand side of 
relation~\ref{eq_ident_decomp} is mapped to 
\vskip10pt
\begin{figure}[h!]
\hair 2pt
\labellist
\pinlabel $s-\lambda$ at 35 0
\pinlabel $s$ at 190 0
\pinlabel $s$ at 215 210
\pinlabel $s-1$ at 200 100
\pinlabel $s-\lambda+1$ at 48 100
\pinlabel $s-\lambda$ at 25 210 
\endlabellist
\centering
\figins{-45}{1.50}{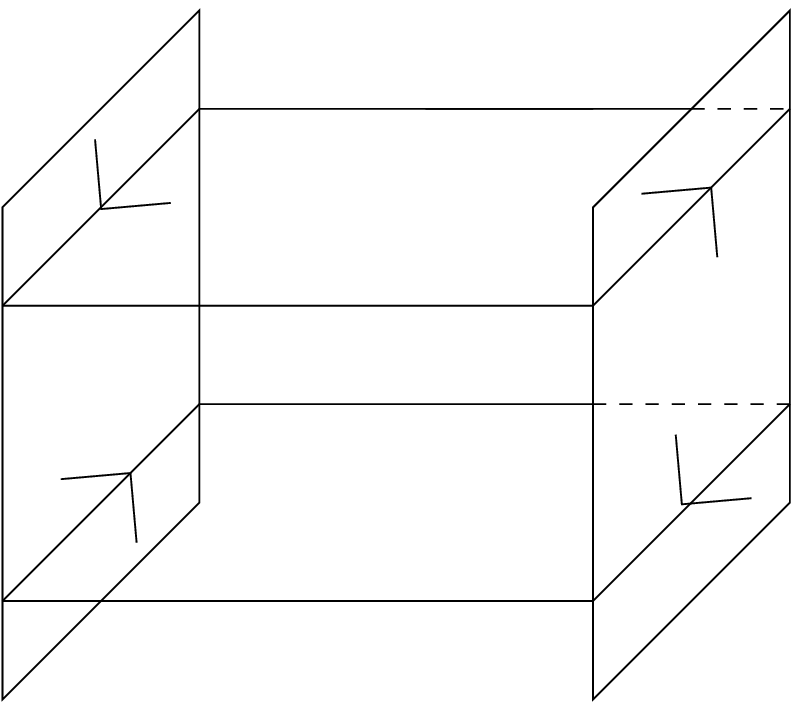}
\end{figure}
\vskip5pt    
\noindent The right-hand side is mapped to 
\begin{figure}[h!]
\hair 2pt
\labellist
\pinlabel $s-\lambda$ at 40 0
\pinlabel $s$ at 190 0
\pinlabel $s$ at 215 210
\pinlabel $s-\lambda$ at 25 210 
\pinlabel $s$ at 645 0
\pinlabel $s-\lambda$ at 495 0
\pinlabel $s$ at 670 210
\pinlabel $s-\lambda$ at 485 210
\pinlabel $+\sum_{f=0}^{\lambda-1}\sum_{g=0}^{f}\sum_{\alpha=0}^{\lambda+1}$ 
at 350 100 
\pinlabel $f-g$ at 575 60
\pinlabel $\alpha$ at 490 100
\pinlabel $g$ at 585 100
\pinlabel $\lambda+1-\alpha$ at 702 100 
\pinlabel $\lambda-1-f$ at 590 160
\endlabellist
\centering
\figins{-45}{1.50}{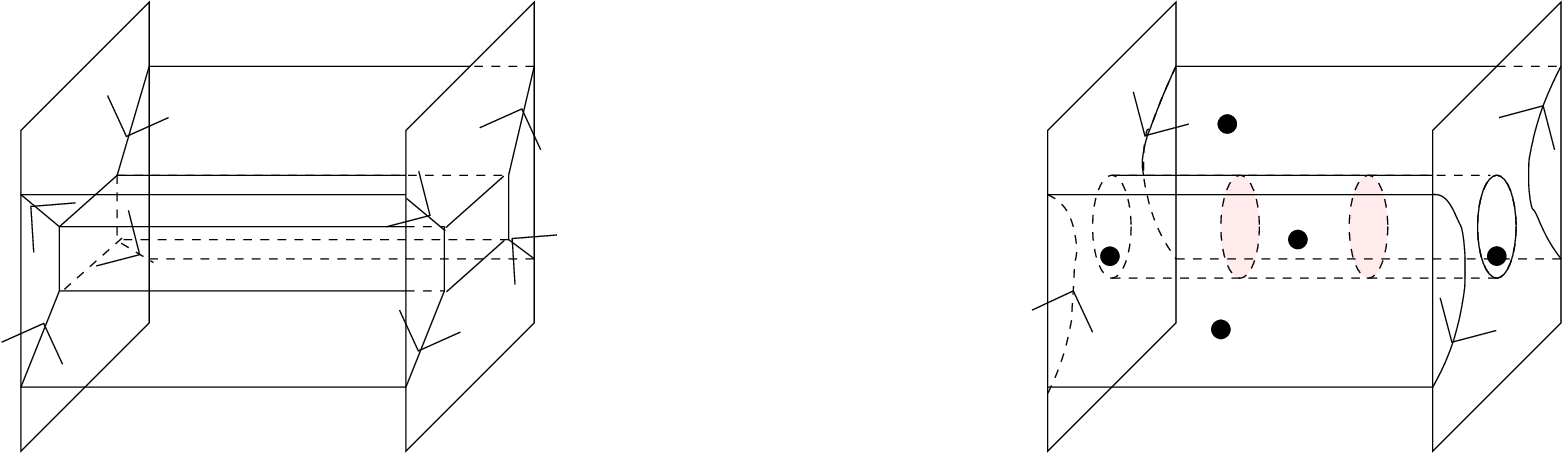}
\end{figure}
\vskip5pt    
\noindent In the figure above we have omitted some orientations to avoid the 
cluttering of lines. In the first foam the orientations can be obtained by 
the rule which says that at each trivalent vertex the edges are all oriented 
inwards or all outwards. The circles in the vertical sheets in the second 
foam are oriented such that neighboring arcs have opposite orientations. 
We have also omitted the labels of some facets. From the orientations of 
the edges and the labels which we have given one can easily compute the 
missing ones. To see that the first relation in~\ref{eq_ident_decomp} is 
preserved one first has to use (RD) on all discs in the tube in 
the second foam above. Then use (3D) and (S) and, if necessary, dot migration. 
After doing that we get the following: for $s=2,\lambda=0$ and $s=3,\lambda=1$ 
relation~\ref{eq_ident_decomp} corresponds to an isotopy between foams. For 
$s=2,\lambda=1$ it corresponds to the (SqR) relation on foams. For 
$s=2,\lambda=2$ and $s=3,\lambda=2$ it corresponds to (DR). 
For $s=3,\lambda=3$ it corresponds to (CN). For all other values of $s$ and 
$\lambda$ the functor $\HH_s$ maps all foams to zero. 

The last type of $\mathfrak{sl}_2$ relations we have to check are the NilHecke 
relations~\ref{eq_nil_rels} and~\ref{eq_nil_dotslide}. Let us first have a look 
at the first relation in \ref{eq_nil_rels}. The diagram is mapped to 
\vskip10pt
\begin{figure}[h!]

\hair 2pt
\labellist
\pinlabel $s$ at 180 0
\pinlabel $s-\lambda$ at 30 0
\pinlabel $s+2$ at 185 200
\pinlabel $s-\lambda-2$ at 5 200 
\endlabellist
\centering
\figins{-45}{1.70}{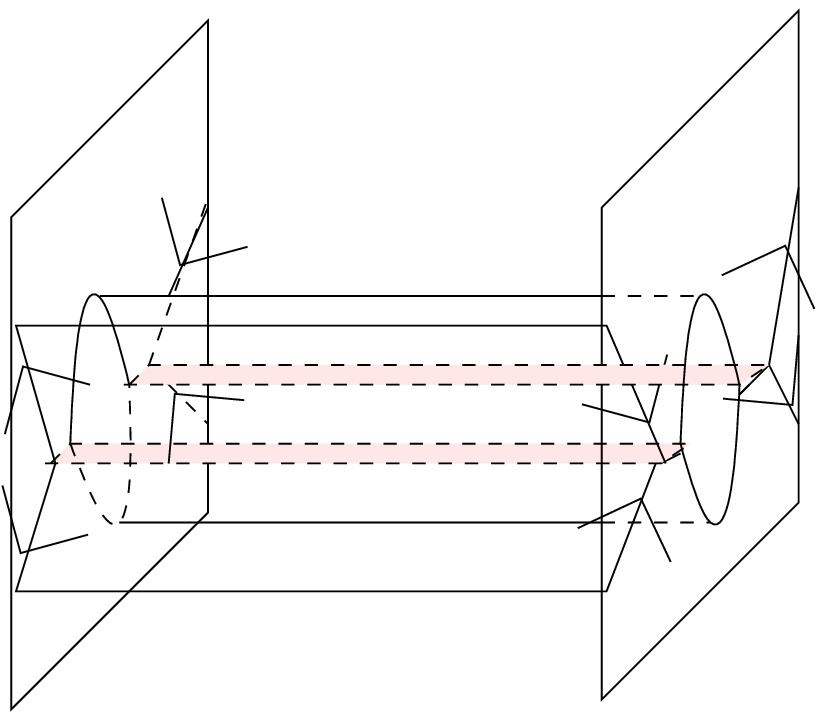}
\caption{The l.h.s. of the first NilHecke relation}
\label{fig:nil}
\end{figure}
\vskip20pt    
\noindent The curves in the middle of each vertical sheet have the same 
orientation as the line segments on that sheet. We have omitted two labels, 
which can be easily computed using 
the orientations. It is easy to see that this foam is always equal to zero. 
Either it is mapped to zero, because some label is out of range, or we get 
a foam with a facet with a bubble without dots, which is equal to zero by 
(RD) and (S). Therefore the first NilHecke relation is always preserved by 
$\HH_s$. 

The second NilHecke relation in~\ref{eq_nil_rels} is always preserved by 
$\HH_s$ because the two sides of the relation are mapped to isotopic foams. 

The third NilHecke relation, which is the one in \ref{eq_nil_dotslide}, 
corresponds to the (DR) relation for all values of $s$ and $\lambda$ 
such that $\HH_s$ is non-zero. The foams can be easily read off from the 
definition of $\HH_s$ above. 

The next relations we have to check are the ones in~\ref{eq_downup_ij-gen}.  
Consider the first one, the second being analogous. 
Suppose $k=3$ and $(i,j)=(1,2)$.
\newpage
\noindent The left-hand side of the equation is mapped to 
\vskip10pt
\begin{figure}[h!]
\hair 2pt
\labellist
\pinlabel $s$ at 355 0
\pinlabel $s+1$ at 370 200
\pinlabel $s-\lambda_1$ at 200 0
\pinlabel $s-\lambda_1-2$ at 170 200
\pinlabel $s-\lambda_1-\lambda_2$ at 55 0
\pinlabel $s-\lambda_1-\lambda_2+1$ at -20 200
\endlabellist
\centering
\figins{-45}{1.50}{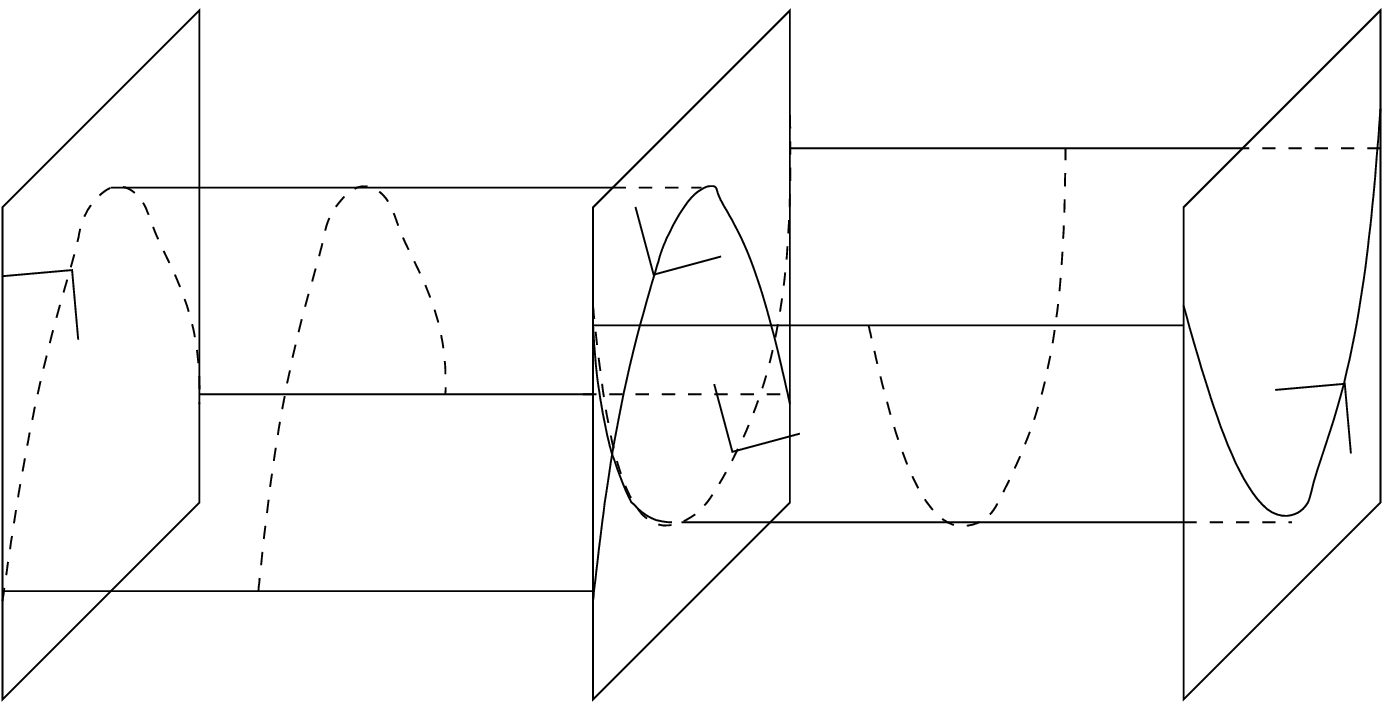}
\end{figure}
\vskip10pt    
\noindent Note that to be completely accurate one would have to draw the 
crossings in the middle vertical sheet as little dumbbles, but that would make 
the picture much harder to read. Note that in that middle sheet the regions 
between the arcs are labelled $s-\lambda_1-1$. Suppose that all labels are 
within the allowed range. Then either $s-\lambda_1=3$ or $s-\lambda_1-2=0$ 
has to hold. Now look closely at the singular curve in the middle sheet and 
you see that the foam can be isotoped to 
\vskip10pt
\begin{figure}[h!]
\hair 2pt
\labellist
\pinlabel $s$ at 355 0
\pinlabel $s+1$ at 370 200
\pinlabel $s-\lambda_1$ at 200 0
\pinlabel $s-\lambda_1-2$ at 170 200
\pinlabel $s-\lambda_1-\lambda_2$ at 55 0
\pinlabel $s-\lambda_1-\lambda_2+1$ at -20 200
\endlabellist
\centering
\figins{-45}{1.50}{foamexample}
\end{figure}
\vskip10pt    
\noindent The case $j=i-1$ is similar. If $|i-j|>1$ the isotopy is 
straightforward, because the horizontal sheets are glued to different 
vertical sheets.  

Next we prove that the $R(\nu)$-relations are preserved. 
The first $R(\nu)$-relation is not too hard. First suppose $k=3$ and 
$(i,j)=(1,2)$. The left-hand side of relation~\ref{eq_r2_ij-gen} is mapped to 
a foam like the one in Figure~\ref{fig:nil}, except that both horizontal sheets have the same 
orientation. Note that in the middle sheet the two arcs have opposite 
orientation. There are now two possible non-zero cases. 
The middle region of the middle 
sheet is labelled $0$ or $3$, and therefore corresponds to a hole, or 
it is labelled $1$ or $2$, in which case it is a disc 
bounded by a singular curve. In the latter case one can apply the (RD) 
relation to the disc and obtain the desired sum of identity foams with dots. 
In the former case there are two singular curves formed by the arcs that do 
not bound the hole. One can apply relation (DR) to obtain the desired sum 
of identity foams with dots. If $i\cdot j=0$, the horizontal sheets in the 
foam are glued to different vertical sheets and we can apply a simple 
isotopy. 

The second set of $R(\nu)$ relations, given in~\ref{eq_dot_slide_ij-gen}, 
is preserved by $\HH_s$ because the $i$-strands before and 
after the crossing are mapped to the same facet in the corresponding foam 
and so are the $j$-strands, since $i\ne j$. Therefore the dots before and 
after the crossing both live on the same facet and the relation is preserved 
because dots can freely move on a facet.    

Finally we have to prove relations~\ref{eq_r3_easy-gen} and 
\ref{eq_r3_hard-gen}. For $i\ne k$ and $i\cdot j=0$, relation 
\ref{eq_r3_easy-gen} clearly corresponds to an isotopy of the corresponding 
foams, because the horizontal sheet corresponding to the $j$-strand is not 
attached to the same vertical sheet as the one corresponding to the 
$i$-strand.  

The proof of relation~\ref{eq_r3_hard-gen} is slightly harder. Suppose 
$i=k$ and $i\cdot j\ne 0$. Note that the 
case $i=j=k$ was already proved above. Here we only prove the 
case $j=i+1$, the case $j=i-1$ being similar. 
Without loss of generality we assume that $k=2$ and 
$i=1$. In the figure below we show the images under $\HH_s$ of the left and 
the right-hand sides of relation~\ref{eq_r3_hard-gen} respectively, which 
gives an equation of foams of the form $A-B=C$.

\begin{figure}[h!]
\hair 2pt
\labellist
\pinlabel $s$ at 810 5
\pinlabel $s-\lambda_1$ at 655 5
\pinlabel $s-\lambda_1-\lambda_2$ at 510 5  
\pinlabel $s$ at 362 5
\pinlabel $s-\lambda_1$ at 205 5
\pinlabel $s-\lambda_1-\lambda_2$ at 60 5  
\pinlabel $-$ at 425 90
\endlabellist
\centering
\figins{0}{1.50}{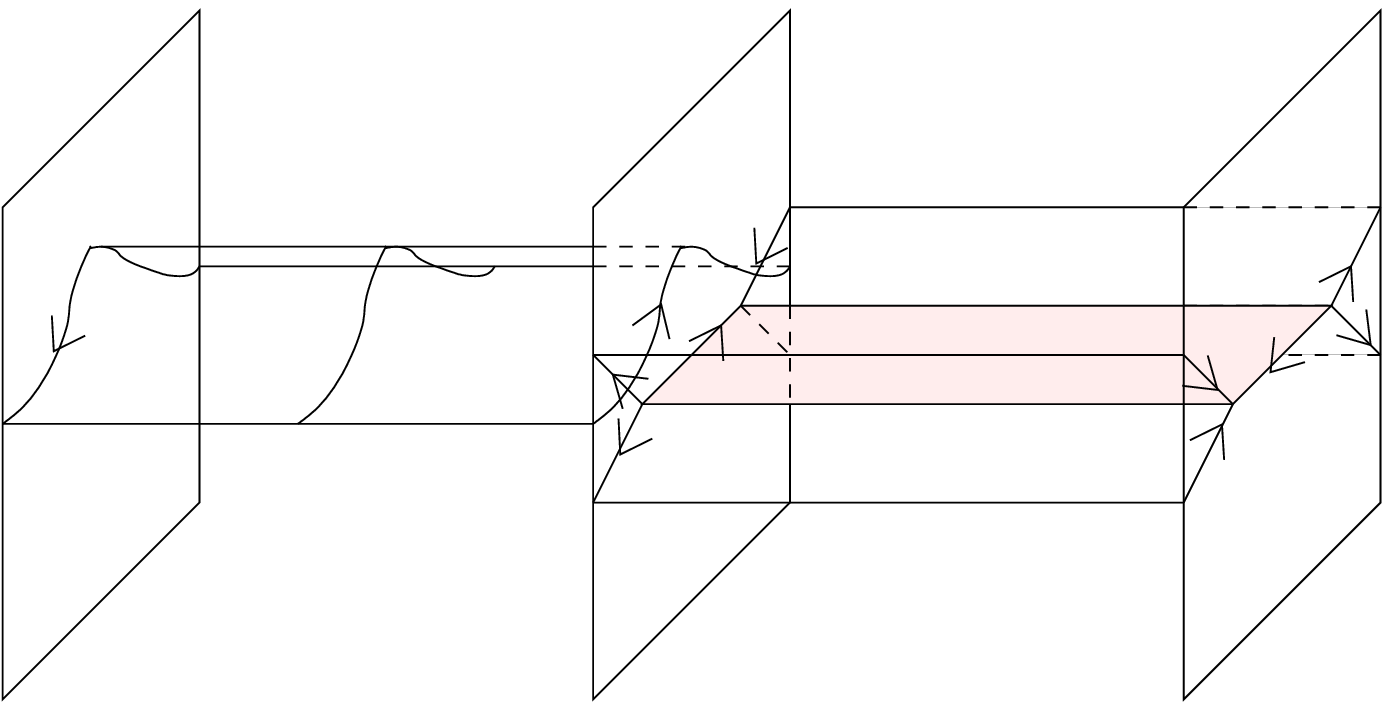}$\qquad$
\figins{0}{1.50}{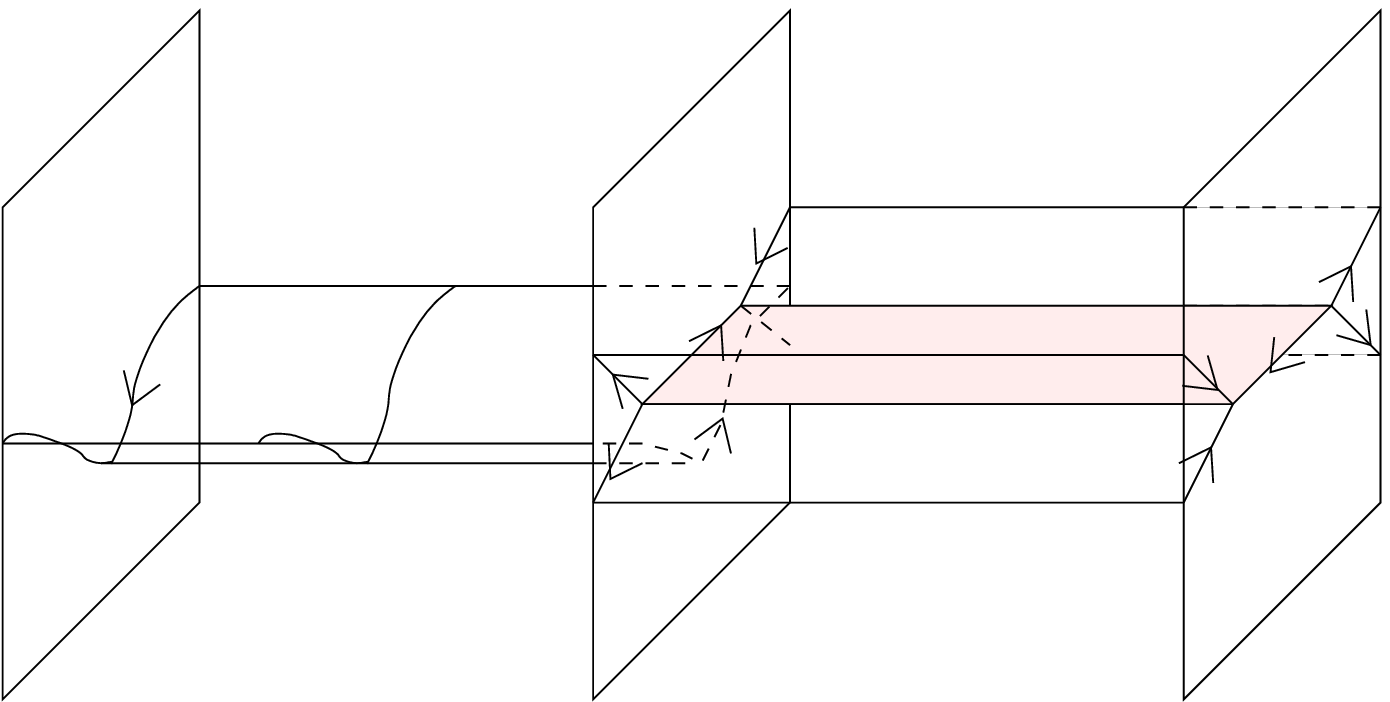}
\end{figure}
\vskip20pt

\begin{figure}[h!]
\hair 2pt
\labellist
\pinlabel $s$ at 360 5
\pinlabel $s-\lambda_1$ at 205 5
\pinlabel $s-\lambda_1-\lambda_2$ at 60 5  
\endlabellist
\centering
\figins{0}{1.50}{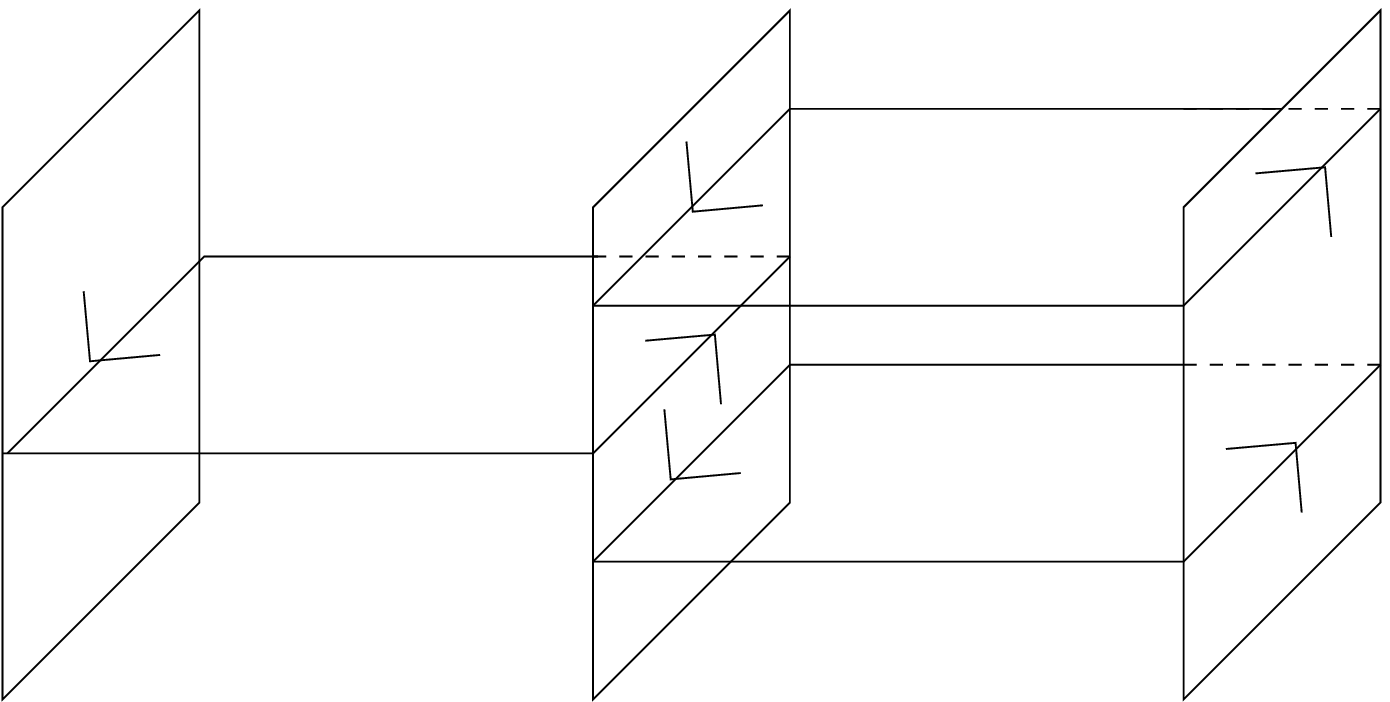}
\end{figure}
\vskip20pt
 
\noindent We only 
show the labels at the bottom. The other labels can be easily computed using 
the orientations of the edges. Note that the square regions 
in the middle of the middle vertical sheet of $A$ and $B$ have label 
$s-\lambda_1-2$ and $s-\lambda_1+1$ respectively. The label at the bottom of 
that sheet is $s-\lambda_1$. Taking into accound the other labels as well, we 
see that there are essentially three 
distinct non-trivial cases: $s-\lambda_1=1$, $s-\lambda_1=2$  and 
$s-\lambda_1=3$ (cases 1,2 and 3). In the first case $s-\lambda_1-2=-1$, so 
$A=0$. To see that $B$ is isotopic 
to $C$ one only has to follow the singular curves. The third case is similar, 
$B$ being zero and $A$ being isotopic to $C$. In the second case both $A$ and 
$B$ can be non-zero. In that case we get the (SqR) relation, which again can 
be seen from looking carefully at the singular curves of $A$, $B$ and $C$.          
\end{proof}


\vspace*{1cm}

\noindent {\bf Acknowledgements} 
I thank Mikhail Khovanov for suggesting this project to me and for his helpful 
input. I thank Mikhail Khovanov and Aaron Lauda for letting me copy their 
LaTeX version of the definition of $\Ucat$ almost literally and I thank them 
and Columbia University for 
their hospitality during my visits in 2008 and 2009. 

Financial support was provided by the 
Funda\c {c}\~{a}o para a Ci\^{e}ncia e a Tecnologia (ISR/IST plurianual funding) through the
programme ``Programa Operacional Ci\^{e}ncia, Tecnologia, Inova\-\c
{c}\~{a}o'' (POCTI) and the POS Conhecimento programme, cofinanced by the European Community 
fund FEDER.



\vspace{0.1in}

\noindent{ \sl \small Marco Mackaay, Departamento de Matem\'{a}tica, Universidade do 
Algarve, Campus de Gambelas, 8005-139 Faro, Portugal  and CAMGSD, 
Instituto Superior T\'{e}cnico, Avenida Rovisco Pais,  
1049-001 Lisboa, Portugal} \newline \noindent 
{\tt \small email: mmackaay@ualg.pt}

\end{document}